\documentclass[onefignum,onetabnum]{siamart190516}

\usepackage{lipsum}
\usepackage{amsfonts}
\usepackage{graphicx}
\usepackage{epstopdf}
\usepackage{algorithmic}
\ifpdf
  \DeclareGraphicsExtensions{.eps,.pdf,.png,.jpg}
\else
  \DeclareGraphicsExtensions{.eps}
\fi

\newsiamremark{remark}{Remark}
\newsiamremark{hypothesis}{Hypothesis}
\crefname{hypothesis}{Hypothesis}{Hypotheses}
\newsiamthm{claim}{Claim}

\headers{Sticky Brownian Motion}{N. Bou-Rabee, M. Holmes-Cerfon}

\title{Sticky Brownian Motion \\ and its Numerical Solution\thanks{
\funding{N.~B.-R. was supported in part by the NSF under Grant No.~DMS-181637. M.H.-C. was supported in part from the Department of Energy Grant DE-SC0012296 and the Alfred P. Sloan foundation.}}}

\author{Nawaf Bou-Rabee\thanks{Department of Mathematical Sciences, Rutgers University Camden, 311 N 5th Street, Camden, NJ 08102, USA
  (\email{nawaf.bourabee@rutgers.edu}).}
\and Miranda Holmes-Cerfon\thanks{Courant Institute of Mathematical Sciences, New York University, 251 Mercer Street,  New York, NY 10012-1185 USA 
  (\email{holmes@cims.nyu.edu}).}}

\usepackage{amsopn}

\usepackage{stmaryrd}
\usepackage{multirow,mathtools,latexsym, mathrsfs, hyperref, pdftricks,tikz,framed,listings,color,mleftright}

\usetikzlibrary{decorations.markings, calc, arrows} 
\usepackage{ifthen}

\usepackage{soul}
\usepackage[shortlabels,inline]{enumitem} 

\newsiamthm{question}{Question}
\newsiamthm{assumption}{Assumption}
\newsiamthm{example}{Example}

\newsiamthm{defn}{Definition}

\definecolor{myGreen}{rgb}{0.0, 0.5, 0.0}

\definecolor{darkblue}{rgb}{0.2, 0.2, 0.7}
\definecolor{darkred}{rgb}{0.7, 0.2, 0.2}

\newcommand{\eps}{\epsilon}
\newcommand{\dd}[2]{\frac{d #1}{d #2}}

\newcommand{\half}{\frac{1}{2}}
\newcommand{\E}{\mathbb{E}}
\newcommand{\R}{\mathbb{R}}

\ifpdf
\hypersetup{
  pdftitle={SBM \& its Numerical Solution},
  pdfauthor={N. Bou-Rabee, M. Holmes-Cerfon}
}
\fi

\begin{document}

\maketitle

\begin{abstract}
Sticky Brownian motion is the simplest example of a diffusion process that can spend finite time both in the interior of a domain and on its boundary. It arises in various applications such as in biology, materials science, and finance. This article spotlights the unusual behavior of sticky Brownian motions from the perspective of applied mathematics, and provides tools to efficiently simulate them. 
We show that a sticky Brownian motion arises naturally for a particle diffusing on $\R_+$ with a strong, short-ranged potential energy near the origin. This is a limit that accurately models mesoscale particles, those with diameters $\approx 100$nm-$10\mu$m, which form the building blocks for many common materials.  
We introduce a simple and intuitive sticky random walk to simulate sticky Brownian motion, that also gives insight into its unusual properties.  In parameter regimes of practical interest, we show this sticky random walk is two to five orders of magnitude faster than alternative methods to simulate a sticky Brownian motion. We outline possible steps to extend this method towards simulating multi-dimensional sticky diffusions.

\end{abstract}





\section{Introduction}

Sticky diffusion processes are solutions to stochastic differential equations (SDE) which can `stick' to, i.e. spend finite time on, a lower-dimensional boundary.  The sticking is reversible, so the process can hit the boundary and leave again, and while on the boundary it can move according to dynamics that are different from those in the interior, even when continuously extended to the boundary.  A simple example is a (root-2) Brownian motion which can stick to the origin, called a sticky Brownian motion, whose forward and backward Kolmogorov equations are identically $\partial_t f =  \partial_{xx}f$ with  boundary condition $\partial_x f|_{x=0} = \kappa \partial_{xx} f|_{x=0}$, where $\kappa \geq 0$ is a parameter measuring how sticky the boundary is. On the other hand Dirichlet ($f|_{x=0} = 0$), Neumann ($\partial_x f|_{x=0} = 0$), Wentzell ($\partial_{xx} f|_{x=0} = 0$) and Robin ($\partial_x f|_{x=0} = \kappa f_{x = 0}$) boundary conditions lead to stopped, reflected, absorbed and elastic Brownian motions, respectively \cite{KaTa1981}.  

Discovered in the 1950s by Feller in an attempt to find the most general behaviour of one-dimensional diffusion processes at a boundary \cite{Feller:1952te, Peskir:2015be}, 
sticky diffusions have been studied in the theoretical mathematical literature for several decades. Probabilists have studied the construction and properties of sticky diffusions via martingales or other representations such as random walks in random environments or interacting particle systems \cite{Feller:1954go,Feller:1957ua,Venttsel:1959gb,Ito:1963wl,Stroock:1971wi,Ku1976B,IkWa1989,Amir:1991wb,Warren:1997wg,nguyen2018sticky,nguyen2019fick,barraquand2019large}, and very recently, sticky diffusions have been used to expand the scope of probabilistic coupling techniques \cite{Ho2007, EbZi2016,Zi2017}.
In PDE theory, analysts have studied well-posedness and semigroup generation for parabolic and elliptic problems with sticky boundary conditions, called generalized Wentzell or Wentzell-Robin boundary conditions in the PDE literature   \cite{luo1991linear,zeng1994linear,apushkinskaya2000survey,apushkinskaya2000venttse,favini2002heat,MR1695147}.

In applications, sticky diffusions arise in a variety of models of physical and natural processes which are naturally described by a set of variables that can change dimension. Examples arise in biology, where
molecules diffuse near a sticky wall or cell membrane \cite{Gandolfi:1985do,Graham:1995uh}; in epidemics, \cite{Calsina:2012ee}, where the concentration of pathogens in an individual can be sticky at concentration level zero; 
in operations research, as a particular limit of storage processes   modeling queues, inventories, insurance risks, etc \cite{Harrison:2016kq}; and in  finance, where the evolution of interest rates can be sticky near zero \cite{Longstaff:1992hv,Kabanov:2007ga}.

Our own interest is in the dynamics of mesoscale particles, those with diameters of $\approx 100nm-10\mu m$, which occur widely and form the building blocks of common materials like paint, toothpaste, concrete, ketchup, and many others \cite{Lu:2013dn}. 
Such particles interact attractively over ranges typically much smaller than their diameters \cite{Manoharan:2015ko,HolmesCerfon:2017hz}. Remarkably, systems with short-ranged interactions are often insensitive to the exact shape of the attractive well of the interaction potential, with most behaviour depending only on one or two parameters such as some combination of the well depth and well width  \cite{Noro:2000bi,Platten:2015dj}. Therefore, it is often effective to model such systems in the \emph{sticky limit}, where the well width is taken to zero, and the well depth to infinity, such that the probability of forming a contact remains constant. In this limit the dynamics of the collection of particles approaches  a sticky diffusion process, with a boundary when a pair of particles are exactly in contact \cite{Baxter:1968dh,Stell:1991va,Fantoni:2006bz,Meng:2010gsa,HolmesCerfon:2013jw,Kallus:2017hi}.  The sticky boundary conditions behave in a similar way to holonomic constraints in molecular dynamics, which eliminate fast bond-length or bond-angle vibrations, and thus reveal the molecular structure more clearly and allow a larger time step in simulations \cite{RyCiBe1977}. 

In the sticky limit, one may be interested in studying the forward and backward Kolomgorov equations to obtain analytical, asymptotic, or numerical solutions that give physical insight, or in simulating the sticky processes themselves to obtain pathwise results. However, neither of these goals is readily attainable: the first, since sticky processes are relatively unknown in the applied math community, techniques to study them are rare 
 and usually invented on a case-by-case basis.   Indeed, by-and-large, applied mathematicians focus on PDEs with classical boundary conditions, and sticky diffusions are beyond this scope, since their transition probability measures have a part that is singular with respect to the Lebesgue measure in the domain.  The second, because there are currently no methods to simulate a sticky diffusion directly: there is no practical way to extend existing methods for discretizing SDEs based on choosing discrete time steps, such as Euler-Maruyama or its variants \cite{KP, Go2001, BoGoTa2004}, to sticky processes; a rough explanation for why is that in these methods one will never hit the boundary exactly. One can approximate a sticky diffusion by a reflecting diffusion with an artificial force near the boundary to encourage the particle to stay there when it gets near, but for a good approximation, the force must be strong and short-ranged. Most SDE solvers are explicit, especially in molecular dynamics applications where evaluating forces is the most costly step, so one must take a prohibitively small timestep to resolve these forces, which unfortunately, severely limits the timescales one can simulate. 

Our aim in this article is twofold: one, we wish to bring the topic of sticky diffusions and their associated PDEs to the attention of the applied math community, and to introduce tools to study them from an applied math perspective. To this end, we show  how a sticky Brownian motion arises as a limit of reflected Brownian motions with a strong short-ranged force at the origin, and discuss the limiting forward and backward equations, which must be used with care because of their unusual boundary conditions (Section \ref{sec:SBM}.)  
Two, we wish to introduce a numerical method to simulate a sticky diffusion, which simulates the process directly without introducing an artificial force, and which allows one to take a relatively large timestep.
The method is based on discretizing the increments of the process in \emph{space}, rather than in time, and constructing a Markov jump process whose generator locally approximates the generator of the sticky diffusion (Section \ref{sec:1dnumerics}.) We derive a Feynman-Kac formula to show that this Markov jump process can be used to solve certain PDEs to second-order accuracy in the spatial step. 
Basic implementations of the two main numerical approximations used in the paper are provided in Appendix \ref{sec:listings}.

Throughout the paper we focus on a one-dimensional sticky Brownian motion, because this illustrates most of the key ideas and differences from traditional diffusion processes with a minimum of technical difficulties. 
We discuss the connection between our approaches to understanding sticky Brownian motion and those taken in the earlier probability literature (Section \ref{sec:1dprobability}.) We expect the methods we introduce to be fully adaptable to higher-dimensional diffusion processes, and outline the steps required to do so in the conclusion (Section \ref{sec:conclusion}.)

\section{Sticky Brownian motion}\label{sec:SBM}


In this section we consider how a sticky Brownian motion (SBM) arises naturally for a particle diffusing in a potential energy landscape that has a strong, short-ranged potential well near the origin.\footnote{In Sections \ref{sec:SBM}-\ref{sec:1dFK}, what we call a sticky Brownian motion is traditionally called a root-2 sticky Brownian motion, because it is scaled by a factor of $\sqrt{2}$ from a traditional Brownian motion. In Section \ref{sec:1dprobability} we use SBM to refer to a  traditional sticky Brownian motion.}  In turn, this motivates the generator and the backward and forward Kolmogorov equations for an SBM, equations and operators which require working in an unusual function space because the transition probabilities will have a singular part.  

\subsection{Setup}

Consider a diffusion process $X^\epsilon_t$  on $\mathbb{R}_{\ge 0}$, depending on a parameter $\epsilon > 0$, which solves
\begin{equation}\label{1d:dX}
dX^\epsilon_t = -\partial_xU^\epsilon(X^\epsilon_t)dt + \sqrt{2}\: dW_t\;, 
\end{equation}
with a reflecting boundary condition at the origin. 
Here $W_t$ is a standard Brownian motion, and $U^\epsilon(x):\R\to \R$ is a  function parameterized by $\epsilon$. If $X^\epsilon_t$ is the position at time $t$ of a particle moving on the real axis, then the force it feels is $-\partial_xU^\epsilon(X^\epsilon_t)$. 
Recall that, if $\int_{0}^{\infty} e^{-U^{\epsilon}(x)} dx < \infty$,
then $e^{-U^{\epsilon}(x)}$ is the non-normalized stationary (or equilibrium) probability density of $X^\epsilon_t$. 

Equation \eqref{1d:dX} is a special case of the \emph{Brownian dynamics equations}, which are a good model for the dynamics of mesoscale particles in a fluid \cite{Frenkel:2001uy}. For a more general system of particles, $X_t^\epsilon$ would be the configuration, a vector of particle positions, $U^{\epsilon}(x)$ would represent the potential energy of a particular configuration (nondimensionalized by temperature), which is usually a sum of the potential energy between each pair of particles, and the equations may additionally include a friction tensor, depending on the configuration and the velocity, modeling hydrodynamic interactions between particles. We specialize to a single scalar equation and ignore friction, but still think of $U^\epsilon(x)$ as a potential energy, representing the energy of the particle $X_t^\epsilon$ as a function of its distance to another particle.  Physically, this equation could be realized by  holding one particle in place at the origin while another is particle moves on a line. 

An important point is that for mesoscale particles, $U^\epsilon(x)\approx 0$ outside an interval that is very narrow compared to the particles' diameters 
\cite{Manoharan:2015ko,HolmesCerfon:2017hz}. For example, for particles interacting with a so-called ``depletion'' interaction \cite{Asakura:1954jy}, the range of the interaction in a typical experiment was estimated to be about 5\% of the particles' diameters \cite{Meng:2010gs}. Therefore, when two particles' surfaces are within this range of each other, they stay close together for a long time, but when they are further apart they don't feel each other at all, and diffuse independently. An even shorter range is achieved by particles that interact via sticky single-stranded DNA coated on their surfaces: here, the range of the interaction is the average radius of a coiled DNA strand, which is typically around 10nm, about 1\% of the diameter of a 1$\mu$m particle \cite{Wang:2015ep}. Such DNA-coated particles are studied extensively because the DNA allows one to code complex interactions between different types of particles, and hence, to program them to assemble into a great many different materials
\cite{Macfarlane:2011fh,Casey:2012hx,Wang:2015ep,Rogers:2016bd,Zhang:2017kw}.

\subsection{Assumptions on the potential energy function}

\begin{figure}\centering
\includegraphics[width=0.5\textwidth]{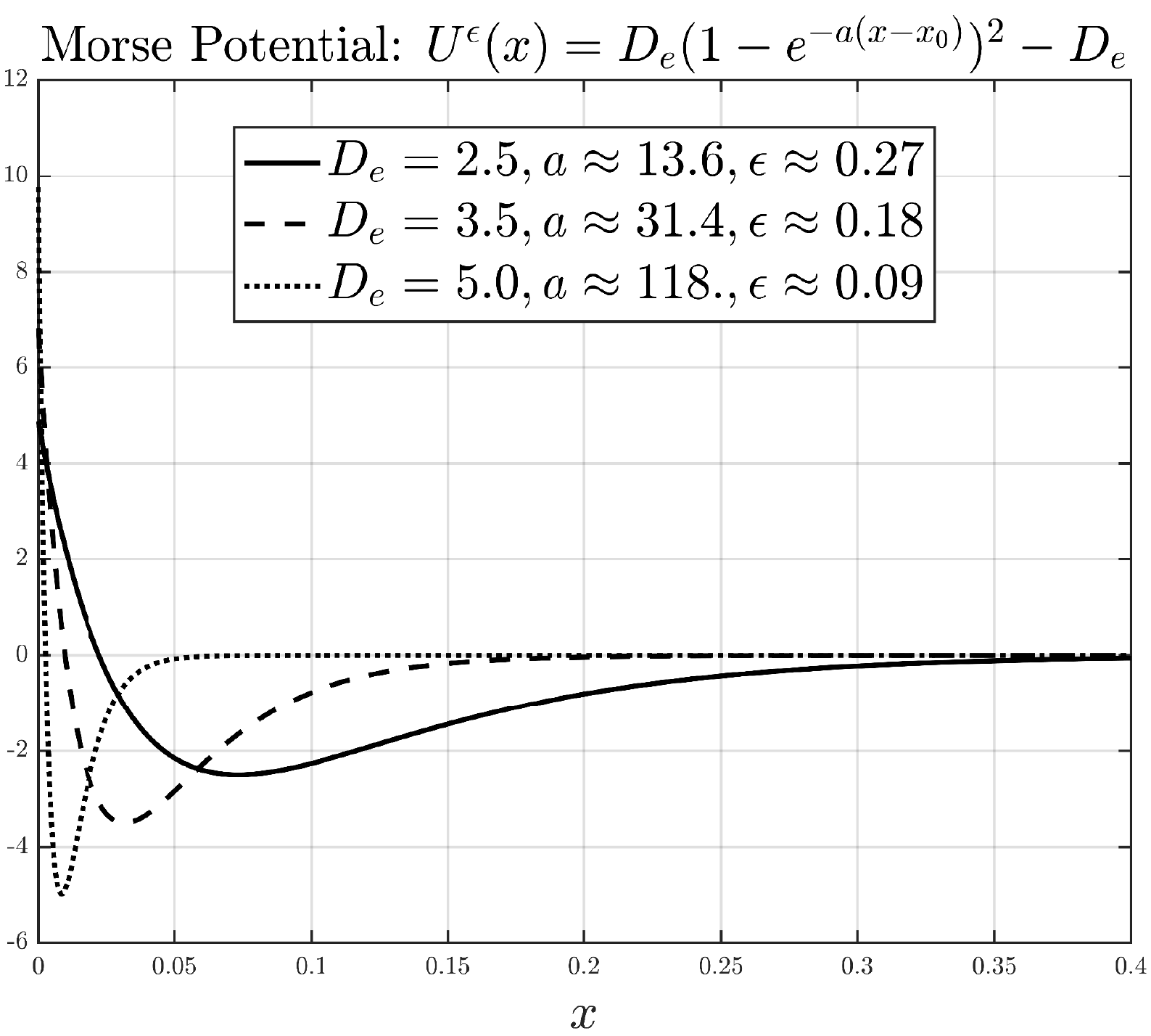} 
\caption{Plots of Morse potential energy functions (see Example~\ref{ex:morse}) with parameters $\kappa = 1$, $D_e$ as indicated in the figure legend,  $a = \sqrt{\pi}e^{D_e} /(\kappa \sqrt{D_e})$,  and $x_0 = 1/a$. This choice of parameters is motivated by the asymptotic condition $\lim_{\epsilon \to 0} \int_0^{\epsilon} e^{-U^{\epsilon}(x)} dx = \kappa$ where $\epsilon = 1/\sqrt{a}$.  The figure illustrates that in this sticky limit the range of the Morse potential shrinks like $\epsilon$, whereas the depth increases but more slowly like $|\log{\epsilon}|$.  
}\label{fig:Ux}
\end{figure}

With these remarks in mind, we consider a family of potential energy functions $(U^\epsilon(x))_{\epsilon>0}$ with a narrow, deep attractive well, which becomes narrower and deeper as $\epsilon \to 0$. 
We consider the dynamics of $X^\epsilon_t$ as $\epsilon \to 0$, and call this the \emph{sticky limit}. 
Specifically, we impose the following assumptions on $(U^\epsilon(x))_{\epsilon>0}$. 

\begin{assumption} \label{A123}
For any $\epsilon>0$,  $U^\epsilon$ is a function in $C^2(\R)$ satisfying:
\begin{itemize}[nosep]
\item[(A1)] $U^\epsilon(x)$, $\partial_xU^\epsilon(x)$, $\partial_{xx}U^\epsilon(x) \le O(\epsilon)$ for $x \ge \epsilon$;
\item[(A2)] $U^\epsilon(x)$ possesses a unique local minimum in $(0, c\epsilon)$ with minimizer $x_0$ and no local maximum for $x>0$;
\item[(A3)] There exists $\kappa \ge 0$ such that
\begin{equation} \label{kappa}
\lim_{\epsilon\to 0} \int_{0}^{\epsilon} e^{-U^{\epsilon}(x)}dx = \kappa\,.
\end{equation}
\end{itemize}
 	 \end{assumption}

\smallskip


We briefly comment on the physical interpretation of (A1)-(A3). 
Assumption (A1) ensures the potential and its first two derivatives are negligible outside of the interval $(0,\epsilon)$, which we call the \emph{boundary layer}. Outside the boundary layer, $X^\epsilon_t$ feels virtually no force and simply diffuses. 

Assumption (A2) ensures that the dynamics in \eqref{1d:dX} has at most one timescale within the boundary layer. 
A typical timescale for a diffusion process is its mean first passage time (MFPT) to overcome an energy barrier $\Delta U$. 
Under Assumption (A2), the barrier to leaving the interval $(0,\epsilon)$ is the depth of the potential at its minimum, 
leading to an MFPT of approximately $\epsilon e^{\Delta U}$ \cite{gardiner}. 
The assumption rules out pathological potentials with many large  oscillations in the boundary layer, which would give rise to longer dynamical timescales in the boundary layer.

Assumption (A3) is the one that gives rise to stickiness at the origin. It requires that the measure of $(0, \epsilon)$ with respect to the weighted Lebesgue measure $e^{-U(x)} dx$, or in the language of physics, the partition function for this interval, approaches a constant. This constant $\kappa$ determines how sticky the origin is -- larger $\kappa$ means the process will spend longer near the origin on average. For this reason we call $\kappa$ the \emph{sticky parameter}.  Applying Laplace asymptotics to \eqref{kappa} shows that 
\[
\sqrt{2\pi}\lim_{\epsilon \to 0}\frac{\epsilon \: e^{-U^\epsilon(x_0)}}{\sqrt{\partial_{xx}U^\epsilon(x_0)}} = \kappa\,,
\]
where $x_0$ is the minimizer of $U^\epsilon(x)$ in $(0,\epsilon)$. 
Ignoring the second derivative shows that the depth must scale very nearly as $|U^\epsilon(x_0)|\sim |\log \epsilon|$, the logarithm of the width of the potential. This implies the timescale computed from the MFPT for leaving the boundary layer  is $\epsilon e^{\Delta U}\sim O(1)$. 
If the scaling of the depth of the well is larger than $|\log \epsilon|$, then large deviation theory would be more appropriate to describe the dynamics of the limiting process \cite{Dembo,grafke2018}. However, if the scaling is smaller, then the limiting process spends no time on the boundary, and simply reflects off of it, as we will see momentarily.

The requirement that $U^{\epsilon}$ has two derivatives is not necessary, but is included to simplify our calculations and avoid dealing with discontinuities in the coefficients of \eqref{1d:dX}. An example of a potential energy function which doesn't satisfy this condition, but which is commonly used to model short-ranged potentials, is the square-well potential.\footnote{Square-well potentials have the form $U^\epsilon(x) = C_\epsilon$  for $x\in [0,\epsilon]$, $U^\epsilon(x) = 0$ for $x > \epsilon$, $U^\epsilon(x) = \infty$ for $x<0$; where $C_\epsilon$ is a constant which depends on $\epsilon$.} By introducing a smooth approximation, we expect our asymptotic results to hold for a square-well potential as well, though we do not pursue this here. 

\medskip

Assumption~\ref{A123} can be verified for two families of potentials frequently used to model attractive interactions between mesoscale particles, the Morse and generalized Lennard-Jones potential energy functions
\cite{Doye:1995tj, Malins:2009dt, Meng:2010gsa,Wales:2010jp,Calvo:2012bw, Zeravcic:2014it}. 


\begin{example}[Morse Potential] \label{ex:morse}  Fix $\kappa \ge 0$, let $\epsilon>0$, and consider the potential energy function defined by 
\[
U^{\epsilon}(x)  = D_e (1-e^{-a (x - x_0)})^2 - D_e
\] 
with parameters $a=1/\epsilon^2$, $x_0 = \epsilon^2$, 
and  $D_e$ defined implicitly via $e^{D_e} \frac{\sqrt{\pi}}{ a \sqrt{D_e }}  = \kappa$.
This potential is illustrated in Figure~\ref{fig:Ux} for several different values of $\epsilon$.  It has a unique global minimum at $x_0 \in (0, \epsilon)$ with depth $U^{\epsilon}(x_0) = -D_e$, and hence, (A2) holds.  
The width of the basin of attraction of the minimum is $O(a^{-1})=O(\epsilon^2)$. Moreover $U^{\epsilon}(x)$ and all of its derivatives are exponentially small for $x \geq \epsilon$, which implies that (A1) holds.  Lastly, as $\epsilon \to 0$, $D_e \to \infty$ with parameters chosen as above, 
straightforward Laplace asymptotics  gives that
$
\lim_{\epsilon \to 0} \dfrac{\int_{0}^{\epsilon} e^{-U^{\epsilon}(x)}dx}{e^{-U^{\epsilon}(x_0)} \frac{\sqrt{2\pi}}{\sqrt{(U^{\epsilon})'' (x_0) }}} =\lim_{\epsilon \to 0} \dfrac{\int_{0}^\epsilon e^{-U^{\epsilon}(x)}dx}{e^{D_e} \frac{\sqrt{\pi}}{ a \sqrt{D_e }}} = 1  \,.
$ Thus, (A3) holds.
\end{example}

\begin{example}[Lennard-Jones($2m$,$m$) potential]
Fix $\kappa \ge 0$, let $\epsilon>0$, and consider the potential energy function defined by 
\[
U^{\epsilon}(x) = D_e\left( \left(\frac{1}{x-x_0 + 1}\right)^{2m}-2\left(\frac{1}{x-x_0 +1}\right)^{m}\right)
\]
with parameters $m=1/\epsilon^2$, $x_0=\epsilon^2$ and $D_e$ defined implicitly via $e^{D_e} \frac{\sqrt{\pi}}{ m\sqrt{D_e }}  = \kappa$.
 This potential has a unique global minimum at $x_0=\epsilon^2$ with depth $U^{\epsilon}(x_0) = -D_e$, and hence, (A2) holds.  The width of the basin of attraction of the minimum is $O(m^{-1})=O(\epsilon^2)$.  Moreover $U^{\epsilon}(x)$ and all of its derivatives are exponentially small for $x \geq \epsilon$, which implies that (A1) holds.  Similar to the preceding example, one can verify that \eqref{kappa} holds.   Hence, Assumption \ref{A123} is satisfied.
\end{example}

\subsection{Dynamics of $X^\epsilon_t$ for $\epsilon \ll 1$}

Consider the dynamics of $X^\epsilon_t$ in \eqref{1d:dX} when $\epsilon$ is small.  When $X^\epsilon_t$ is far enough away from the origin, it feels no force, and simply diffuses, like a Brownian motion.  When $X^\epsilon_t$ is within a distance of $\epsilon$ from the origin, it feels a strong force keeping it near the minimum of $U^\epsilon$ for a long time, until an occasional large fluctuation pushes it out of the range of the force. 

How long does it stay near the origin, and does this time remain significant as $\epsilon\to 0$? We start by computing the mean first-passage time (MPFT) to escape from a region near the origin. 

\begin{lemma}\label{lem:taueps}
Let $\tau^{\epsilon}(x) = \mathbb{E} \left( \inf\{ t \ge 0 \mid X^\epsilon_t > \ell, ~X_0 = x \in [0, \ell] \} \right)$ be the MFPT of $X_t^{\epsilon}$ out of $[0, \ell]$  for some $\ell>0$ with initial condition $x \in [0, \ell]$. Then
\begin{equation} \label{eq:MFPTlim}
\lim_{\epsilon \to 0} \tau^{\epsilon}(0) = \kappa \ell  + \frac{\ell^2}{2} \;.
\end{equation}
\end{lemma}

When $\kappa =0$, we recover the MFPT of a reflecting Brownian motion starting at $0$. The time scales as the distance squared, $\tau^{\epsilon}(0) \sim O(\ell^2)$, a traditional diffusive scaling. When $\kappa \to \infty$, the MFPT is infinite,  consistent with the MFPT of an absorbing Brownian motion starting at $0$. 
For intermediate $\kappa$, the MFPT at the origin scales  as $\tau^{\epsilon}(0)\sim O(\ell)$, which is a \emph{ballistic} scaling -- slower (for small $\ell$) than the diffusive scaling. Therefore, we expect the limiting probability density near the origin to be correspondingly large. 

\begin{proof}
The MFPT $\tau^{\epsilon}(x)$ satisfies the boundary value problem 
\[
\partial_x ( e^{-U^{\epsilon} } \partial_x \tau^{\epsilon} ) = - e^{-U^{\epsilon} }  ~~\text{on $[0, \ell]$ with b.c. $\partial_x \tau^{\epsilon}(0) = 0$ and $\tau^{\epsilon}(\ell) = 0$.}
\]  
By integrating twice, the semi-analytic solution to this equation is given by 
\begin{align*}
\tau^{\epsilon}(x) = \int_x^\ell \int_0^r e^{ U^{\epsilon}(r) - U^{\epsilon}(s)} ds dr 
\end{align*} 
and at $x=0$ we obtain, 
\begin{align*}
\tau^{\epsilon}(0) &= \int_0^{\epsilon} e^{-U^{\epsilon}(s)} \int_s^{\epsilon} e^{U^{\epsilon}(r)} dr ds + 
\int_0^{\epsilon} e^{-U^{\epsilon}(s)} \int_{\epsilon}^{\ell} e^{U^{\epsilon}(r)} drds  \\
& \qquad + \int_{\epsilon}^{\ell} e^{-U^{\epsilon}(s)} \int_{s}^{\ell} e^{U^{\epsilon}(r)} dr ds \;.
\end{align*}
Assumptions~\ref{A123} (A1) and (A3) imply that the last two terms converge to the right-hand side of \eqref{eq:MFPTlim}, while Assumption~\ref{A123} (A3) 
implies the first term converges to 0.\footnote{
To see this last point, let $x_0$ be the minimizer of $U(s)$ on $[0,\epsilon]$, and let $u_* = \max_{s\in[x_0,\epsilon]} U(s)$ be the maximum of the energy to the right of the minimum. 
By Assumption~\ref{A123} (A2), there is at most one point $x_*\in[0,x_0)$ such that $U(x_*) = u_*$, and since the energy monotonically increases as $s$ decreases below $x_*$, we have $-(U(s)-U(r))\leq 0$ for $r\geq s$, $s\leq x_*$. 
For $s>x_*$, $r\geq s$, we have $-(U(s)-U(r))\leq -U(s) + u_*$. 
Therefore by Assumption~\ref{A123} (A3), 
\[
\left|\int_0^{\epsilon} e^{-U^{\epsilon}(s)} \int_s^{\epsilon} e^{U^{\epsilon}(r)} dr ds\right| \leq 
\int_0^{\epsilon}  \int_s^{\epsilon} e^{0\vee (-U(s)+u_*)} dr ds  \to 0 \quad\text{ as } \epsilon \to 0\,.
\]
} 
\end{proof}

\begin{figure}\centering
\includegraphics[width=0.6\textwidth]{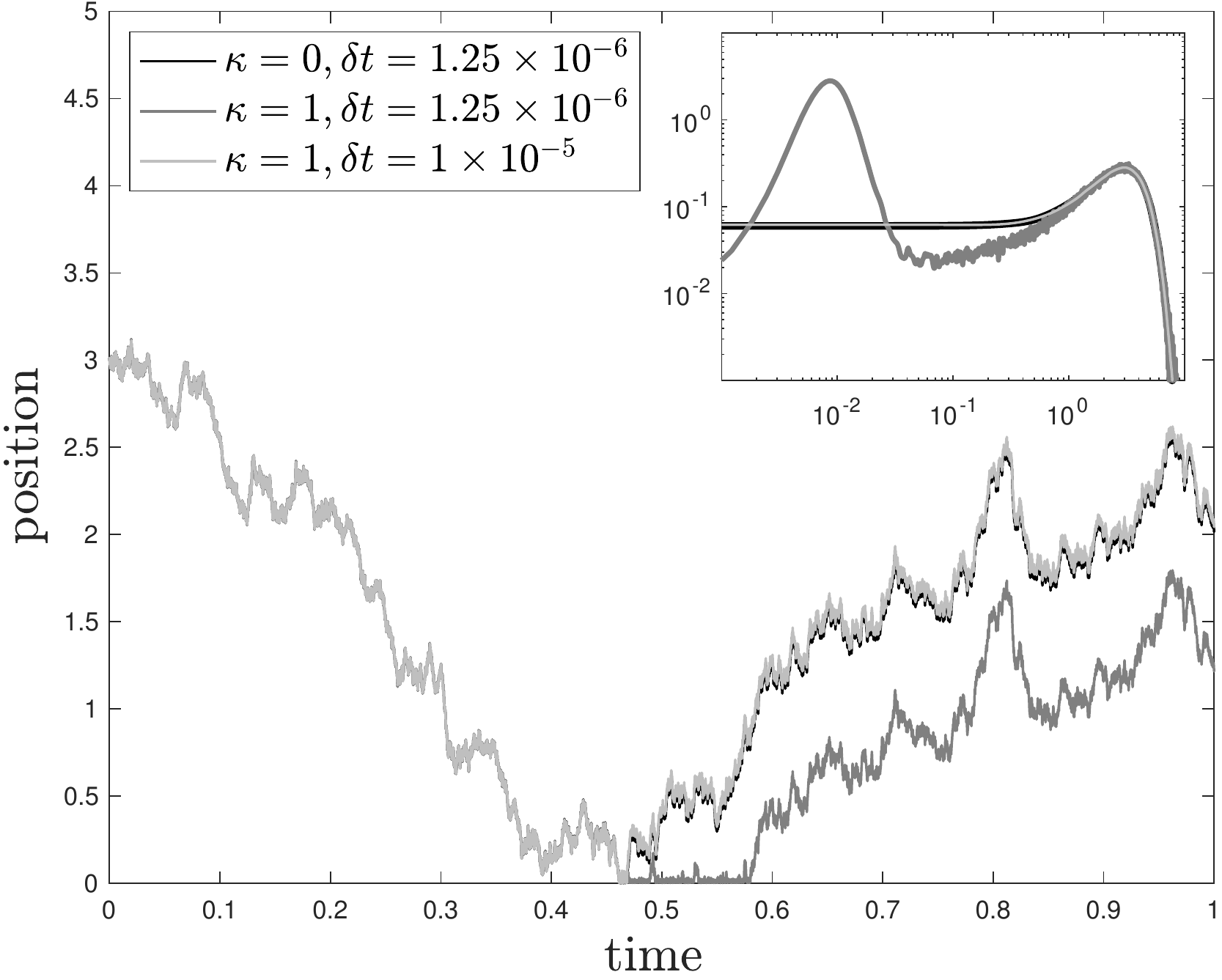} 
\caption{ Realizations of symmetrized Euler-Maruyama  \eqref{1d:symm_euler}  driven by the same realization of Brownian motion with initial condition $\tilde X_0=3$, and $U^{\epsilon}(x)$ defined as the Morse potential in Example~\ref{ex:morse} with sticky parameter $\kappa=1$, potential depth $D_e=5$, and potential range $a\approx118$.  The case $\delta t =1.25 \times 10^{-6}$ (in dark grey) corresponds to a converged numerical solution.  However, when $\delta t=10^{-5}$ (in light grey), the trajectory is indistinguishable from a pure reflecting Brownian motion (in black).  Corresponding empirical densities at $t=1$ are shown in the inset; note the empirical density for $\delta t=10^{-5}$ (in light grey) is on top of the empirical density of the pure reflecting Brownian motion (in black). This shows that the time step size requirement for qualitatively correct solutions of \eqref{1d:symm_euler} is stringent.  For a quantitative test, see Figure~\ref{fig:accuracy_sbm_1d}.  
}\label{fig:symm_euler_sbm_paths}
\end{figure}

\begin{figure}\centering
\includegraphics[width=0.5\textwidth]{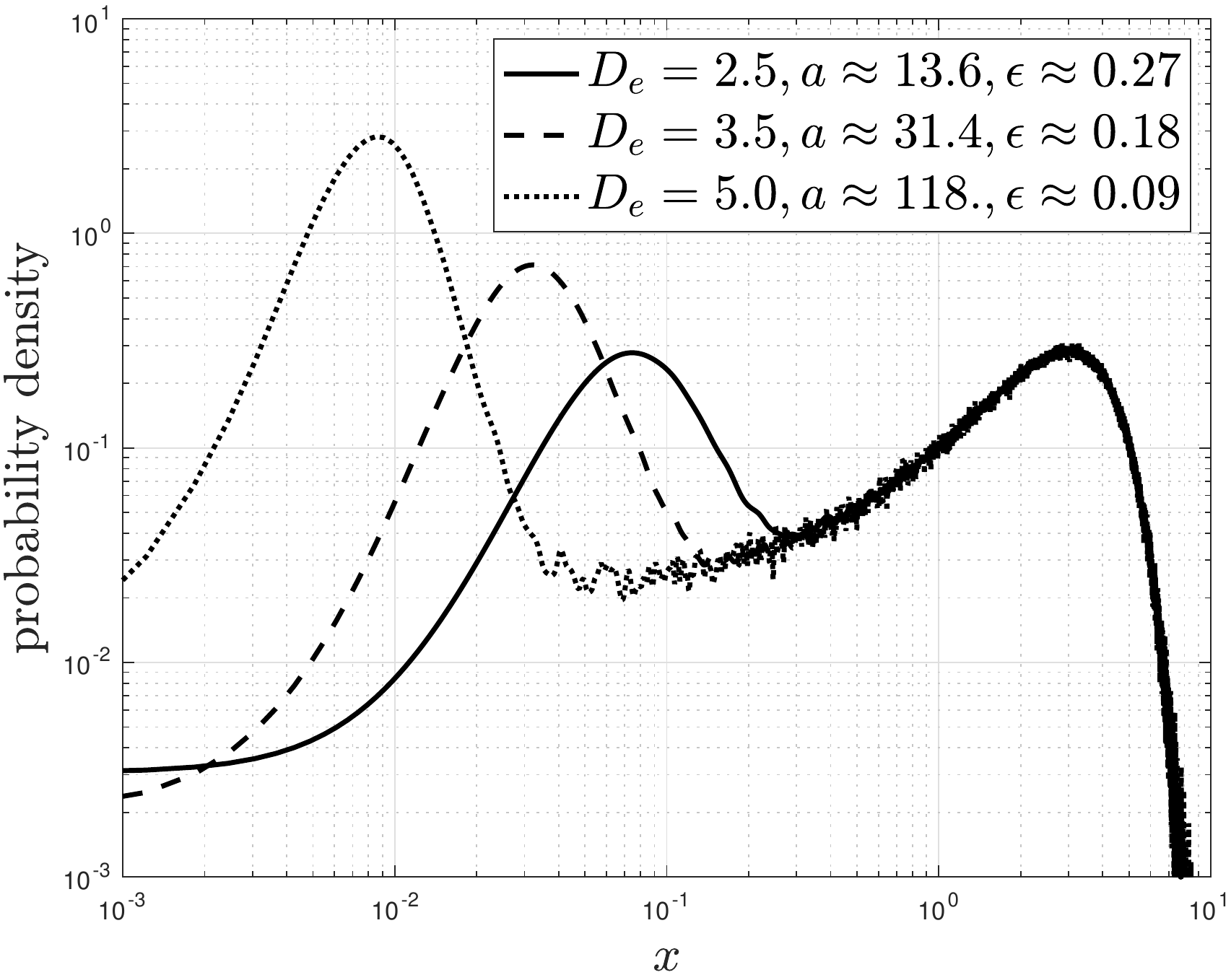} 
\caption{ A plot of the solution to \eqref{1d:FP} with a point mass initial condition at $x=3$ and $U^{\epsilon}$ defined as the Morse potential given in Example~\ref{ex:morse} with sticky parameter $\kappa=1$ and potential depths $D_e$ as indicated in the figure legend.  The horizontal axis is a log scale and shows the probability densities change rapidly near $0$, but are otherwise slowly varying.   The small kinks in the plot are due to Monte Carlo error.  }\label{fig:symm_euler_sbm_pdf}
\end{figure}

To gain further insight into the dynamics, consider the evolution of the probability density $p^{\epsilon}(x,t)$ of  $X^\epsilon_t$. 
It evolves according to the Fokker-Planck equation  
\begin{equation}\label{1d:FP}
\partial_t p^\epsilon(x,t) + \partial_x j^{\epsilon}(x,t) = 0 \,,
\end{equation}
where $j^{\epsilon}(x,t) = - \partial_x U^{\epsilon}(x) p^{\epsilon}(x,t) - \partial_{x} p^{\epsilon}(x,t)$ is the associated probability current, or probability flux. 
The equation is to be solved in the domain $x>0$, $t>0$ with boundary condition $j^{\epsilon}(0,t) = 0$ for all $t \ge 0$, and given initial condition $p^{\epsilon}(x,0)=\phi(x)$. 

For this one-dimensional problem we could compute $p^\epsilon(x,t)$ to high accuracy by solving the PDE \eqref{1d:FP} numerically. However, such an approach won't work in the high dimensions characteristic of systems of particles, so to illustrate the difficulties that may arise, we simulate $10^6$ trajectories of \eqref{1d:dX} numerically and estimate the density by kernel density estimation. 
We use a Morse potential for $U^\epsilon$ as in Example~\ref{ex:morse}, and a symmetrized Euler-Maruyama approximation to solve \eqref{1d:dX}, a method that is first-order weakly accurate method for domains with a smooth enough boundary \cite{Go2001, BoGoTa2004}. 
 The method works as follows: given a time-step size $\delta t$, we set $t_k=k\delta t $ for $k \in \mathbb{N}_{\ge 0}$, and compute the approximation $\tilde X_k$ to $X^\epsilon(t_k)$ as
 \begin{equation} \label{1d:symm_euler}
\tilde X_{k+1} = | \tilde X_k - \partial_xU^\epsilon(\tilde X_k) \delta t + \sqrt{2} (W(t_{k+1}) - W(t_k)) | \;, \quad \tilde X_0 = x \ge 0 \;.
\end{equation} 
See Appendix \ref{sec:listings} for Matlab code implementing this method. 
Unfortunately, when the Morse potential used is short-ranged and strong, an accurate approximation requires a small time step size: when $\kappa=1,a \approx 118$ the time step size required for a qualitatively correct solution is approximately $\delta t \le 2.5 \times 10^{-6}$ --  Figure~\ref{fig:symm_euler_sbm_paths} shows that larger timesteps fail to see the boundary layer entirely.


Figure~\ref{fig:symm_euler_sbm_pdf} plots $p^{\epsilon}(x,1)$  on a log scale for various values of $\epsilon$, and shows that as $\epsilon$ decreases,  $p^{\epsilon}(x,1)$ has an increasingly large peak in density near the origin. Although the width of the peak decreases, the total probability in the peak remains nearly constant with $\epsilon$. Therefore, we expect the probability density to contain a singularity of the form $\delta(x)$ when $\epsilon\to 0$. 

Furthermore, notice that $p^{\epsilon}(x,1)$ is slowly varying for $x \ge \epsilon$ and rapidly varying for $0 \le x \le \epsilon$.  This suggests using the method of matched asymptotic expansions to study \eqref{1d:FP} as $\epsilon \to 0$ \cite{Hinch1991, KevorkianCole1996}. This approach, which is very similar to boundary-layer theory in fluids, proceeds by finding a local `inner' solution near the origin and a local `outer' solution far enough away from the origin, and then matching these local solutions in an intermediate region.

\subsection{Asymptotics of the probability density in the sticky limit} 

We pursue the method of matched asymptotic expansions to show that the leading-order probability density of $X^\epsilon_t$ (leading-order in the sense of a measure) is 
\begin{equation}\label{1d:dens}
\rho(x,t) = p(x,t)(1+\kappa\delta(x))\,,
\end{equation}
where $p(x,t)$ solves
\begin{equation}\label{1d:FPlim}
\partial_t p = \partial_{xx} p\,,\qquad\text{with b.c. }\quad  \kappa \partial_{xx}p = \partial_x p \quad \text{ at } x=0\,,
\end{equation}
with a given initial condition $p(x,0) = \phi(x)$. 
We assume that $\phi$ is continuously differentiable  on $[0,\infty)$ and that $\phi$ and $\phi'$ are bounded on $[0,\infty)$; assumptions which guarantee that \eqref{1d:dens} is well-posed \cite[Theorem 1]{Peskir:2014}. 
Equations \eqref{1d:dens}, \eqref{1d:FPlim} are the main results of this section, and describe the evolution of probability of a SBM \cite{Feller:1952te}. 


To derive these equations we adopt a formal approach to highlight the main ideas, but expect that the argument could be turned into a more rigorous proof of weak convergence by the interested reader.

Let us make the ansatz that the solution to \eqref{1d:FP} away from the origin ($x \ge \epsilon$), or outer solution,  and near the origin ($x\le \epsilon$), or inner solution, have the asymptotic expansions, respectively,
\begin{align}\label{asymptotic_expansion}
p^\epsilon(x,t) \sim \begin{cases} p_0(x,t) + p_1(x,t) + \cdots, & x> \epsilon \,,\\
q_0 (x,t) + q_1(x,t) + \cdots, &  x \le \epsilon\,,
\end{cases}
\end{align}
where $p_0(x,t)$ and $q_0(x,t)$ are the leading order terms in the expansions, and $p_1\ll p_0$, $q_1\ll q_0$ as $\epsilon \to 0$.  
We do not assume any particular scaling for these terms, nor the leading-order solutions; it will turn out that $p_0\sim O(1)$ but $q_0\sim O(\epsilon^{-1})$, because of Assumption~\ref{A123} (A3). 

 By Assumption~\ref{A123} (A1), 
 the outer solution $p_0(x,t)$ satisfies the following linear PDE to leading order as $\epsilon\to 0$:
\begin{equation}\label{1d:p0}
\partial_t  p_0 =  \partial_{xx}  p_0\,,\qquad  x\ge \epsilon\,,
\end{equation}
with $p_0(x,0) = \phi(x)$. 
Near the origin, the density changes rapidly so it is convenient to change variables to $X = x/\epsilon$. Keeping the leading-order terms gives 
\begin{equation}\label{1d:q0eqn}
0 = \partial_X(\partial_XU^\epsilon( \epsilon X)q_0) +  \partial_{XX} q_0\,,\qquad X \le 1
\end{equation}
with the reflecting boundary condition $\partial_X U^{\epsilon}(0,t) q_0(0,t) + \partial_X q_0(0,t) = 0$. 
We have used  Assumption \ref{A123} (A2) and the timescale it implies to argue that $\partial_t q_0$ is lower order than the terms retained in \eqref{1d:q0eqn}.


Equation \eqref{1d:q0eqn} is the equation for the stationary density of a particle diffusing in a potential $U^\epsilon$ with a reflecting boundary condition at the origin.  
The solution is 
\begin{equation}\label{1d:q0}
q_0(\epsilon X,t) = a(t)e^{-U^\epsilon(\epsilon X)}\,, \qquad X \le O(1) 
\end{equation}
where $a(t)$ is some unknown function of time, to be determined by matching to the outer solution.  

Now we match the outer solution $p_0(x,t)$ and the inner solution $q_0(\epsilon X,t)$ at $x=\epsilon$, $X=1$. Unlike the traditional method of matched asymptotic expansions, we do not match in an overlap region, but rather at a single point, which is possible because the perturbation is not singular (the diffusion terms do not change with $\epsilon$.) 
We match using two conditions: one, probability is continuous, and two, probability is conserved. 
The first condition requires that 
\[
p_0(\epsilon,t) = q_0(\epsilon,t) 
\quad\implies\quad
a(t) = p_0(0,t)
\]
to leading order in $\epsilon$,
where we used that $e^{-U^{\epsilon}(\epsilon)} \sim 1$, by Assumption~\ref{A123} (A1), and $p_0(\epsilon,t) \sim p_0(0,t)$.
The condition that probability be conserved requires that 
\[
\dd{}{t}\left( \int_{0}^{\epsilon} q_0(x,t)\: dx + \int_\epsilon^\infty p_0(x,t) \: dx\right) = 0 
\quad \implies \quad
\kappa a'(t) - \partial_x p_0(0,t) = 0\,,
\]
to leading order in $\epsilon$. 
We moved the time derivative into the integrals, 
substituted for $\partial_t p_0$ using \eqref{1d:p0}, and substituted $\kappa$ using  Assumption~\ref{A123} (A3).

Putting these results together gives a boundary condition for the outer solution $p_0$ at the origin, which can be written in two ways: 
\begin{equation}\label{1d:bdy}
\kappa \partial_t p_0|_{x=0} = \partial_xp_0|_{x=0}
\quad\iff\quad
\kappa\partial_{xx}p_0|_{x=0} = \partial_xp_0|_{x=0}\,.
\end{equation}
Combining with \eqref{1d:p0} and removing the subscript on $p$ gives \eqref{1d:FPlim}. 

To obtain \eqref{1d:dens}, notice that by continuity of probability, the leading-order density is $p_0(x,t)e^{-U^\epsilon(x)}$. Assumption \ref{A123} implies that $e^{-U^\epsilon(x)}$ converges weakly to $1+\kappa\delta(x)$ as $\epsilon\to 0$, giving a limiting density of \eqref{1d:dens}.

\subsection{Fokker-Planck equation}

We make a few remarks concerning the limiting dynamics \eqref{1d:FPlim} and their relationship to a sticky Fokker-Planck equation. 

Notice that, as expected, the limiting probability density \eqref{1d:dens} is singular with respect to Lebesgue measure on $\mathbb{R}^+$. Its structure shows that an SBM can spend finite time on any interval in its domain, as well as at the origin, $\{0\}$, sets with intrinsically different dimensions. 
This is an unusual property for diffusion processes; it does not occur for 
 processes with the more typically-studied Dirichlet, Neumann, or Robin boundary conditions. 

It turns out that the measure $1+\kappa\delta(x)$ in \eqref{1d:dens} is infinitesimally invariant,
and is proportional to the invariant measure if the SBM is confined to a compact space, as the following lemma shows. 

\begin{lemma}
Suppose the function $p$ in \eqref{1d:dens}, \eqref{1d:FPlim} additionally satisfies a reflecting boundary condition $p_x = 0$ at $x=L>0$. Then the unique invariant probability measure $\pi$ is  
\begin{equation}\label{1d:pi}
\pi(x) = Z^{-1}( 1+\kappa\delta(x))\,, \qquad \text{where} \quad Z = \kappa + L\,.
\end{equation}
\end{lemma}
This is also the measure obtained as the weak limit of the stationary densities for \eqref{1d:FP}. 
\begin{proof}
Solve for the steady-state solution of \eqref{1d:FPlim} to obtain $p(x,t) = c$ for some constant $c\in \mathbb R$, and use \eqref{1d:dens} to obtain \eqref{1d:pi}.
\end{proof}

It is sometimes helpful to work with a formulation of \eqref{1d:FPlim} directly in terms of the probability density. 
This is possible by first writing the density $\rho$ as a sum of densities on the  different manifolds in its support, as 
\begin{equation}\label{1d:densdecomposition}
\rho(x,t) = \rho_0(t)\mu_0 + \rho_1(x,t)\mu_1\,,
\end{equation}
where $\mu_0(dx)$ is the delta-function measure on $\{0\}$, and $\mu_1(dx)$ is the Lebesgue measure on $(0,\infty)$. We must impose a `continuity condition'
\begin{equation}\label{1d:cty}
    \rho_0(t) = \kappa\rho_1(0,t)
\end{equation}
to be consistent with the asymptotic derivation. 
Calculating $\partial_t\rho$ using \eqref{1d:FPlim} and \eqref{1d:dens} gives the system of equations
\begin{equation}\label{1d:FPfull}
\begin{array}{rll}
  \partial_t \rho_0  = & \partial_x\rho_1 & \text{for }x=0\,, \\
    \partial_t \rho_1 = & \partial_{xx}\rho_1 & \text{for } x\in (0,\infty)\,.
\end{array}
\end{equation}
System \eqref{1d:FPfull} and the continuity condition \eqref{1d:cty} can be interpreted as the Fokker-Planck equation for the evolution of the probability density $\rho$. If we write the system formally as $\partial_t \rho = \mathcal L^* \rho$ for some linear operator $\mathcal L^*$, then \eqref{1d:FPfull} shows that $\mathcal L^*$ should, formally at least, be interpreted as being a different partial differential operator, for the densities on each different manifold in the support of the process. 
Indeed, it turns out that for higher-dimensional sticky diffusions one can  impose different dynamics in the interior of a domain and on the boundary, and these dynamics don't have to bear any relation to each other \cite{IkWa1989}. 


The system \eqref{1d:FPfull} gives a more physical interpretation of the boundary condition in \eqref{1d:FPlim}. It shows the boundary condition simply balances fluxes: the rate of change of the probability at the origin, $\partial_t \rho_t|_{x=0} = \partial_t (\kappa p)|_{x=0}$, equals the flux of probability which leaves the open interval $(0,\infty)$ on the left, $\partial_x\rho_1|_{x=0} = \partial_x p|_{x=0}$. 
Here again we see how the sticky boundary condition interpolates between a reflecting condition, when $\kappa=0$ and the condition is $\partial_x p|_{x=0}$, and an absorbing one, when $\kappa\to\infty$, and the condition approaches $\partial_t p|_{x=0} = 0$.

\medskip

\subsection{Generator}

We now consider how to obtain the generator $\mathcal L$ of a SBM, starting from the Fokker-Planck equation \eqref{1d:FPfull}.
The generator forms the basis for our numerical method, and is in fact the more fundamental quantity describing a Markov process. Recall that the generator is the operator $\mathcal L$ which is the formal adjoint of $\mathcal L^*$, i.e. $\langle f,\mathcal L^*\rho\rangle = \langle \mathcal L f, \rho\rangle$, for all densities $\rho$ of the form \eqref{1d:dens}, and all test functions $f\in C^2_c([0,\infty))$ that satisfy an appropriate boundary condition at zero (to be determined.) 
Because properly defining $\mathcal L^*$ is somewhat subtle due to the singularities at the origin, 
we find it more transparent to work with a weakly equivalent formulation, which asks that $\mathcal L$ satisfy
\begin{equation}\label{1d:weakinner}
\langle f,\partial_t\rho\rangle = \langle \mathcal L f, \rho\rangle\,.
\end{equation}
To proceed, compute the left-hand side of \eqref{1d:weakinner} as 
\begin{align*}
\langle f,\partial_t\rho\rangle &= (f\kappa\partial_t p)|_{x=0} + \int_0^\infty f\partial_tp \: dx \\
&=  (f\kappa\partial_{xx} p)|_{x=0} - (f\partial_x p)|_{x=0} +(p\partial_xf)|_{x=0}  + \int_0^\infty p \partial_{xx}f \: dx\\
&=  (p\partial_xf)|_{x=0} - (\kappa p\partial_{xx}f)|_{x=0} + \int \rho\partial_{xx}f\:dx\,.
\end{align*}
In the first step we substituted $\partial_{xx}p$ for  $\partial_t p$ and integrated by parts, assuming a decay condition at $x=\infty$, and in the second step we used the boundary condition on $p$ at $x=0$ and rewrote the integral in terms of $\rho$. 

We see that \eqref{1d:weakinner} is satisfied if we choose the generator and its associated boundary conditions to be
\begin{equation}\label{1d:generator}
\mathcal L f = \partial_{xx} f, \qquad \text{with b.c. } \quad \kappa\partial_{xx}f = \partial_x f \quad \text{ at } x=0\,.
\end{equation}
For a SBM, the generator and its formal adjoint are equal. 

Notice that if in \eqref{1d:weakinner} we had replaced $\rho$ with $p$, the function solving \eqref{1d:FPlim}, we would not have found appropriate boundary conditions for $f$. This shows the importance of interpreting the weak formulation of \eqref{1d:FPlim} in the correct function space.

From the generator we obtain another way to verify that \eqref{1d:pi} is an infinitesimally invariant measure, by showing $\langle \mathcal L f, \pi\rangle =0$ for all  test functions $f$ satisfying the appropriate boundary conditions. Calculations very similar to the above show this equation holds. 

\subsubsection{Example: MFPT of a sticky Brownian motion} 
As a simple application of \eqref{1d:generator}, we directly calculate the MFPT of an SBM.
\begin{lemma} \label{l:mfpt_sbm}
The MFPT of an SBM out of $[0,\ell]$ with initial condition $x \in [0,\ell]$ is $\tau(x) =  - \kappa x  - (1/2) x^2 +  \kappa \ell  + (1/2) \ell^2$.  In particular, $\tau(0) = \kappa \ell + (1/2) \ell^2$.
\end{lemma}

\begin{proof}
The MFPT of an SBM satisfies the boundary value problem  $\mathcal L \tau(x) = -1$, $\tau(l) = 0$, plus any other boundary conditions associated with $\mathcal L$. Specifically, 
\[
\partial_{xx} \tau(x) = -1 ~~ \text{on $[0, \ell]$ with b.c. $\kappa \partial_{xx} \tau(0) = \partial_x \tau(0) $ and $\tau(\ell) =0$}\,.
\] 
The solution is $\tau(x)=-(1/2) x^2 + c_1 x + c_2$ where $c_1$ and $c_2$ are constants determined by imposing the boundary conditions.
\end{proof}

Note that $\tau(0)=\lim_{\epsilon\to 0}\tau^\eps(0)$, the limit of the MFPT for the reflecting diffusions, as in Lemma \ref{lem:taueps}. 

\medskip

\subsubsection{Example: transition rates between sticky points} 
As another application of \eqref{1d:generator}, we consider a Brownian motion on a line segment with sticky endpoints, and calculate the transition rates between the endpoints. 
Such a setup is a good model for the transition paths between clusters of mesoscale particles, and  quantitatively predicts their transition rates 
\cite{Perry:2015ku}. 
Suppose the endpoints have sticky parameters $\kappa_1,\kappa_2$ and the line has length $L$, as shown below. 
\begin{center}
\begin{tikzpicture}[scale=3]
\draw (0,0) -- (1,0) ;
\node at (0,0.13) {$\kappa_1$};
\node at (1,0.13) {$\kappa_2$};
\node[circle,draw, fill=black!50, inner sep=0pt, minimum width=7pt,label=below:{$L$}] at (1,0) {};
\node[circle,draw, fill=black!50, inner sep=0pt, minimum width=7pt,label=below:{$0$}] at (0,0) {};
\end{tikzpicture}
\end{center}
The generator for this problem is
\begin{equation}\label{stickyline}
\mathcal L_L f = \partial_{xx} f \quad \text{on } [0,L], 
\qquad \text{with b.c. }  \begin{cases} \kappa_1\partial_{xx}f = \partial_x f &\text{ at } x=0\\
\kappa_2\partial_{xx}f = -\partial_x f & \text{ at } x=L
\end{cases}
\end{equation}
The sign of $\partial_x f$ is reversed for the boundary condition at $L$, because the probability flux is in the opposite direction from the flux at $0$. 

A framework for calculating transition rates between disjoint sets $A$, $B$ exactly is given by Transition Path Theory \cite{E:2010hs}. 
These transition rates can be determined from empirical averages by using the following limit relations
\[
k_{A\to B} = \lim_{T\to \infty}\frac{N_T}{T_A}, \qquad 
k_{A\to B} = \lim_{T\to \infty}\frac{N_T}{T_B}\,,
\]
where $T$ is the time of observation, $N_T$ is the total number of transitions observed from $A$ to $B$ in time $T$, and $T_A, T_B$ are the total times during which the process last hit $A,B$ respectively; they satisfy $T_A+T_B=T$. 
We apply Transition Path Theory to the sets $A=\{0\}$ and $B=\{L\}$ to compute the quantities above, and refer the reader to \cite{E:2010hs} for more justification of these calculations. 

\begin{lemma} 
Given a process with generator $\mathcal L_L$ defined in \eqref{stickyline}, and let $A=\{0\}$, $B=\{L\}$. Then the stationary distribution is $\pi(x) = \frac{1}{L+\kappa_1+\kappa_2}(1 + \kappa_1\delta(x) + \kappa_2\delta(x-L))$, and the transition rates are
\[
k_{A\to B} = \frac{1}{\kappa_1 L + L^2/2}\,,\qquad 
k_{B\to A} = \frac{1}{\kappa_2 L + L^2/2}\,.
\]
\end{lemma}

It is worth noting that the transition rates depend inversely on $L$, the total distance that the process must diffuse. This is a contrast to theories of transition rates derived from the Arrhenius formula or Transition State Theory, for which the transition rates depend only on local properties of critical points, such as the energy difference between critical points and their curvatures \cite{gardiner}. In the problem above, the energy differences are contained solely in $\kappa_1,\kappa_2$. 

\begin{proof}
One can verify that  $\langle \mathcal L_L f, \pi\rangle = 0$ for all twice-differentiable functions $f$ satisfying the boundary conditions, so $\pi$ given above is the stationary distribution. Now we turn to calculating the transition rates. We first calculate the committor function $q(x)$, a function that gives the probability of hitting $B$ first before $A$ starting from $x$. It solves the boundary value problem
\[
\mathcal L_L q = 0, \qquad q(A) = 0, \quad q(B) = 1\,,
\]
whose solution is  $q(x) = x/L$. 
Next we calculate the overall rate of transition $\nu = \lim_{T\to \infty} \frac{N_T}{T}$, as $\nu = \pi(x)\partial_x q(x)$ for any $x\in (0,L)$, giving $\nu = L^{-1}(L+\kappa_1+\kappa_2)^{-1}$. 
Then we calculate the ``reactive probabilities'' $\rho_A = \lim_{T\to\infty} \frac{T_A}{T}$, $\rho_B=\lim_{T\to\infty} \frac{T_B}{T}$. These are computed as $\rho_A = \int_0^L \pi(x)(1-q(x))dx$, $\rho_B = \int_0^L \pi(x)q(x)dx$, giving 
$\rho_A = (\kappa_1 + L/2) / (L + \kappa_1 + \kappa_2)$, $\rho_B = (\kappa_2 + L/2) / (L + \kappa_1 + \kappa_2)$. 
Finally, the reaction rates are computed as $k_{A\to B} = \nu / \rho_A$, $k_{B\to A} = \nu / \rho_B$, giving the result above. 
\end{proof}

\section{Numerical method to simulate a sticky Brownian motion}\label{sec:1dnumerics}

It is not obvious from either the Fokker-Planck equation or the generator how one should simulate trajectories of a sticky Brownian motion. 
One option is to return to the derivation as a limit as $\epsilon \to 0$ of $X^\epsilon_t$, and choose a small $\epsilon$ and simulate \eqref{1d:dX}. However, our earlier calculations (Figure \ref{fig:symm_euler_sbm_paths}) showed that for small $\epsilon$ one must use a very small timestep;  for large $\epsilon$ the solution will be inaccurate. 

Therefore, we turn to an entirely different method, based on constructing a continuous-time Markov chain whose generator $Q$ approximates the generator $\mathcal L$ of an SBM. Specifically, we spatially discretize the infinitesimal generator and boundary conditions of the SBM, for example using a finite difference approximation, to obtain the generator of a Markov jump process on the set of discretization points.  This Markov jump process may be simulated using a simple Monte Carlo method known variously as the Stochastic Simulation Algorithm, kinetic Monte Carlo, the Doob-Gillespie algorithm or the Gillespie algorithm \cite{Do1942,Do1945,Gi1976,Gi1977}. 

These approximations of sticky diffusions go back to the Markov Chain Approximation Method (MCAM) invented by Harold Kushner in the 1970s to approximate optimally controlled diffusion processes \cite{Ku1970, Ku1973, Ku1975, Ku1976A, Ku1976B, Ku1977, Ku2001, KuDu2001, Ku2002,Ku2004, Ku2006, Ku2011}. 
However, because of their interest in stochastic control problems, these works mainly focus on numerical solutions with gridded state spaces and  use numerical linear algebra to construct an approximation.   In the physics literature, a Monte-Carlo method was developed to construct the approximation \cite{elston1996numerical, WaPeEl2003,Ph2008,MeScVa2009,LaMeHaSc2011}, an approach which seems to go back to at least \cite{elston1996numerical}.  
More recently, a new ``gridless'' framework was introduced for constructing Markov jump process approximations for diffusions \cite{BoVa2018}, which allows the domain of the diffusion process to be unbounded, does not require that the diffusion process is symmetric, and does not assume that the infinitesimal covariance matrix of the diffusion process is diagonally dominant.  This generalization allows the jump size of the numerical solution to be uniformly bounded, which makes it easier to numerically treat boundary conditions.  Very recently, these ideas have been extended to SPDE problems \cite{bou2018spectrwm}. 

\subsection{Numerical Algorithm} 

To apply this approach, let us discretize the interval $[0,\infty)$ into a set of grid points $\mathcal R^h = \{0,h,2h,\ldots\}$, where $h$ is the spacing between neighboring grid points. 
Let $f\in C^4_b(\mathbb{R})$ and let $f_0,f_1,\ldots$ be the values of the function at the grid points, i.e. $f_k = f(kh)$ as illustrated below.
 \begin{center}
 \begin{tikzpicture}[scale=1.5]
\draw[-, thick](2,0.0) -- (6,0.0);
\draw[-, thick](3,-0.2) -- (3,0.2);
\node[black,scale=1.] at (2,-0.35) {$-h$};
\node[black,scale=1.] at (3,-0.35) {$0$};
\node[black,scale=1.] at (4,-0.35) {$h$};
\node[black,scale=1.] at (5,-0.35) {$2 h$};
\node[black,scale=1.] at (2,0.35) {$f_{-1}$};
\node[black,scale=1.] at (3,0.35) {$f_0$};
\node[black,scale=1.] at (4,0.35) {$f_1$};
\node[black,scale=1.] at (5,0.35) {$f_2$};
\node[black, scale=1.5,fill=white] at (6.0,0.0) {$\dotsc$};

\foreach \x in {2.0,3.,4.,...,5.} 
{
\ifthenelse{\lengthtest{\x pt > 2pt} \AND \lengthtest{\x pt < 8pt}}{\filldraw[color=black,fill=black] (\x,0) circle (0.12);}{\filldraw[color=black,fill=white] (\x,0) circle (0.12);}
};

\end{tikzpicture}
\end{center}
The white dot indicates a ghost grid point, which we recruit in our construction.

At each interior grid point, the generator  $\mathcal L f$ (see \eqref{1d:generator}) can be approximated by
\begin{equation}\label{1d:Lfdisc}
(\mathcal L f)(k h) = \frac{f_{k+1} - 2f_k + f_{k-1}}{h^2} + O(h^2)\,,\qquad k\geq 1\,.
\end{equation}
This approximation is a second-order centered finite difference approximation to the second derivative operator $\partial_{xx}$. 
At the boundary grid point, $k=0$, we do not know the value of the `ghost' point $f_{-1}$ that is needed in \eqref{1d:Lfdisc}, so we solve for it using the boundary condition in \eqref{1d:generator}.  The discretized  boundary condition is 
\[
\kappa \frac{f_{1} - 2f_0 + f_{-1}}{h^2} = \frac{f_1 - f_{-1}}{2 h} + O(h^2) \Leftrightarrow f_{-1} = \frac{4\kappa f_0 + (h-2\kappa)f_1}{h+2\kappa} + O\left(\frac{h^4}{h+2 \kappa}\right) 
\] 
where again we used a second-order centered finite difference scheme to evaluate $\partial_x f$.  Using this equation to eliminate $f_{-1}$ in \eqref{1d:Lfdisc} with $k=0$ gives an approximation to the generator at the boundary as
\begin{equation}\label{1d:Lfdiscbdy}
(\mathcal L f)(0) = \frac{2f_1 - 2f_0}{h^2+2\kappa h} + O\left(\frac{h^2}{h+2 \kappa}\right)  \,.
\end{equation}
This local error estimate suggests that this approximation is globally second-order accurate in $h$ when $\kappa>0$.  When $\kappa=0$, which corresponds to a reflecting Brownian motion, the remainder term in \eqref{1d:Lfdiscbdy} may lead one to surmise that the approximation is only first-order spatially accurate.  However, it turns out that even in this case, the global error is still second-order in $h$, because roughly speaking, the local error at the boundary is also proportional to the mean occupation time of a reflecting Brownian motion near the boundary, which is $O(h)$.

Now, we construct a continuous-time Markov chain $Y_t$ with state space $\mathcal R^h$, whose generator $Q$ is the discrete approximation to $\mathcal L $.  That is, we set  
\begin{equation}\label{1d:Ldisc}
(Q f)_k = \begin{dcases}
\frac{f_{k+1} - 2f_k + f_{k-1}}{h^2} , &  k=1,2,\ldots,\\
 \frac{2f_1 - 2f_0}{h^2+2\kappa h} , & k=0.
\end{dcases} 
\end{equation}
The infinite matrix associated to the generator $Q$ has nonzero entries $Q_{kj} = 1/h^2$ if $j=k\pm 1$ and $k\geq 1$, $Q_{kk} = -2/h^2$ if $k\geq 1$, $Q_{01} = 2/(h^2+2\kappa h)$, $Q_{00} = -2/(h^2+2\kappa h)$.  One can verify that $Q$ is the generator of a continuous-time Markov chain, since the coefficients $Q_{kj}$ with $k\neq j$ are nonnegative, and $\sum_{j}Q_{kj} = 0$.  
Note that when $\kappa =0$ in \eqref{1d:Ldisc} one obtains the generator of a reflecting random walk.

Realizations of this sticky random walk $Y_t$ can be simulated exactly using a simple Monte Carlo method \cite{Do1942,Do1945,Gi1976,Gi1977}. Suppose $Y_{t_0}=kh$ at time $t_0$. The process is updated in two steps. 
\begin{itemize}[nosep]
\item Pick a state $\nu\in\{(k+1)h,(k-1)h\}$ with probabilities proportional to $Q_{k,k+1}$, $Q_{k,k-1}$ respectively. 
Specifically, for $k=1,2,\ldots$, we jump to the left or right states with equal probability, and for $k=0$, we always jump to $k=1$. 
\item Choose a random time $\tau \sim \mbox{Exp}(-Q_{kk})$, where $\mbox{Exp}(\lambda)$ is the distribution of an exponential random variable with  mean $1/\lambda$. Set $Y_{t_0+t} = kh$ for $0<t<\tau$ and set $Y_{t_0+\tau} = \nu$. 
That is, the process jumps to a new state $\nu$ after a random time $\tau$. 
\end{itemize}

 To sum up, the process $Y_t$ is a continuous-time random walk, a ``sticky random walk'' (SRW), with mean waiting time $h^2/2$ at interior grid points and  $\kappa h + h^2/2$ at the sticky boundary. See Appendix \ref{sec:listings} for Matlab code implementing this method.

 Realizations of $Y_t$ for $\kappa = 1$ and different values of $h$ are plotted in Figure \ref{fig:SBMsims}. 
Away from 0, the process is a random walk with random waiting times, which looks increasingly like a Brownian motion as $h\to 0$. When the process hits 0, however, it spends much longer there on average than it does at interior points. Overall the process looks like a Brownian motion which has been ``slowed down'' near 0, an observation that can be made rigorous as we discuss in Section \ref{sec:1dprobability}.

\begin{figure}\centering
\includegraphics[trim={4.5cm 0 0.9cm 0},clip,width=1\linewidth]{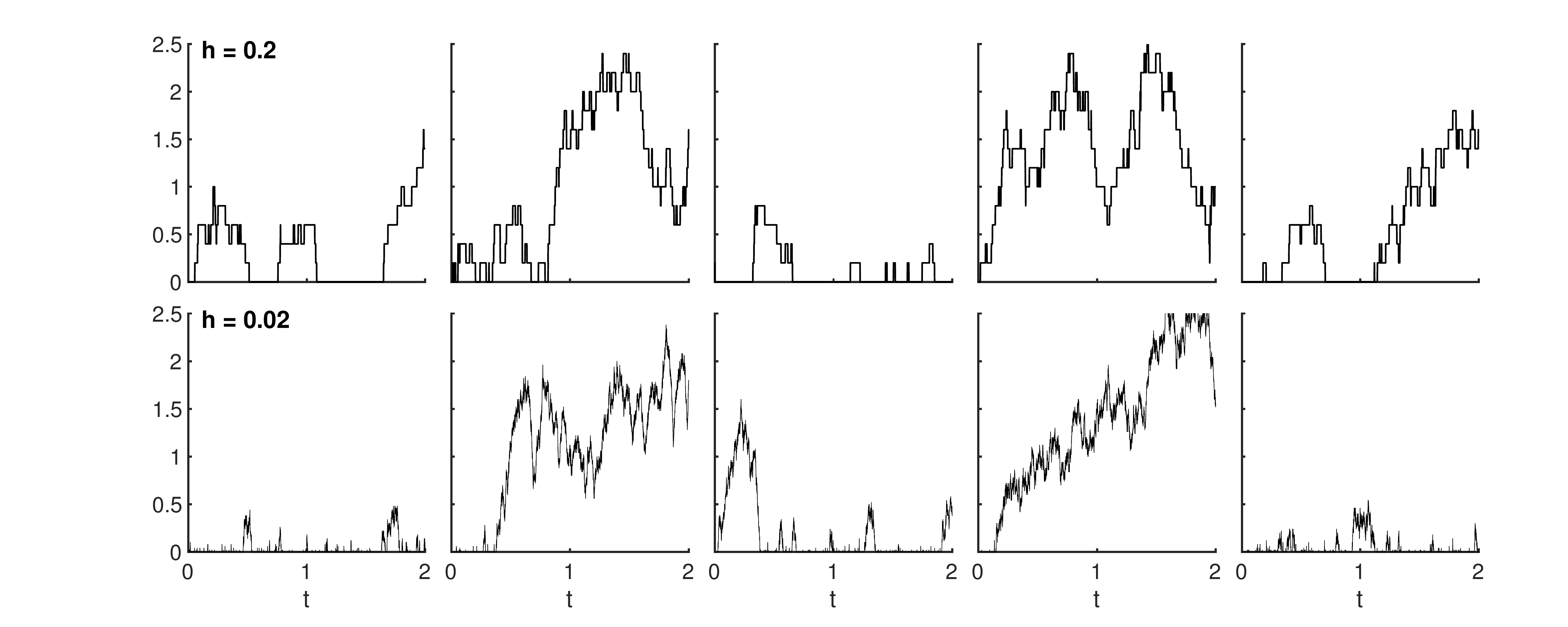}
\caption{ Independent realizations of the sticky random walk $Y_t$ with $\kappa=1$ and with space step $h=0.2$ (top), $h=0.02$ (bottom), over a time interval of length 2. As $h$ decreases, realizations of $Y_t$ approach realizations of a sticky Brownian motion. }
\label{fig:SBMsims}
\end{figure}

 \subsection{Properties of the numerical solution}
 
 This difference in the scaling with $h$ of the waiting times at boundary and interior points, nicely illustrates the unusual behavior of an SBM at the origin and gives insight into why it is sticky there.  The mean waiting time at interior points is $h^2/2$, the usual scaling for diffusion motion, which requires that $\E (\Delta Y)^2 / \tau \to cst$, where $\Delta Y=\pm h$ is the jump at each step of the algorithm. The mean waiting time at the origin is $\kappa h +  h^2/2 = \kappa h + o(h)$ -- longer by a factor of about $h$, so the scaling is \emph{ballistic} at the origin.  The amount of time the process spends at the origin is the same as it would spend at about $\kappa/h$ interior grid points, which is the number one would use to discretize an interval of length $\kappa$. Therefore, we expect we expect this SWR to spend a finite, non-zero amount of time at the origin as $h\to 0$, an amount of time which increases in proportion to $\kappa$.

 Interestingly, although the SRW spends finite time at the origin as $h\to 0$, in the limit it never spends a whole interval of time at the origin. This is evident from the numerical solution, since the waiting time at the origin approaches zero as $h\to 0$. This implies that every time an SBM hits the origin, it leaves right away, just as does a Brownian motion; yet somehow the total measure of the points at which it equals zero is positive. 

We have seen the ballistic scaling  of times near the origin before, in Lemmas \ref{lem:taueps} and \ref{l:mfpt_sbm}. Indeed, 
the mean holding time of the sticky random walk $Y_t$ at $0$ is identically the MFPT of a SBM with $\ell=h$ (see Lemma~\ref{l:mfpt_sbm}.)  The additive property of the MFPT then implies the following lemma.

\begin{lemma} \label{l:mfpt_Yt}
Given $h>0$ and $\ell \in \mathcal{R}_h$, the MFPT of $Y_t$ out of $[0,\ell]$ with initial condition $x \in [0,\ell] \cap \mathcal {R}_h$ is $\tau(x) =  - \kappa x  - (1/2) x^2 +  \kappa \ell  + (1/2) \ell^2$.  In particular, $\tau(0) = \kappa \ell + (1/2) \ell^2$.
\end{lemma}

Another useful property of $Y_t$ is that once a simulation has been performed for some fixed $\kappa$, one can obtain a trajectory for any other value $\kappa'\neq \kappa$ simply by scaling the holding time of the jumps which leave the origin, since the probabilities of jumping to each state do not change; see  Figure~\ref{fig:1dnumerical} for an illustration. 
This is a useful property as it allows one to investigate multiple values of $\kappa$ with one single simulation. 

\begin{figure}\centering
\includegraphics[width=0.65\textwidth]{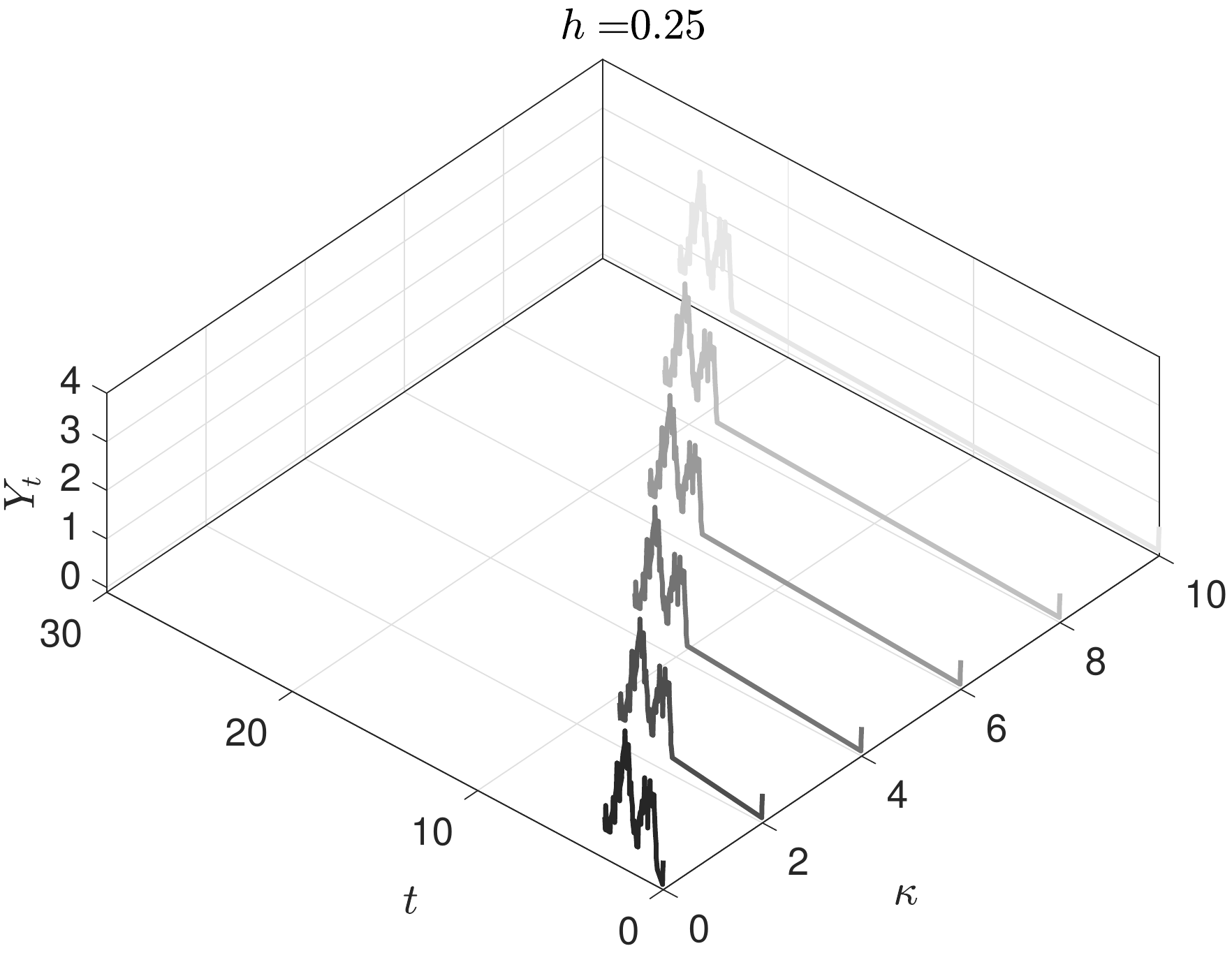} 
\caption{ Realizations of the sticky random walk with $h=0.25$ and varying $\kappa$.  Each realization is producing by running the Monte Carlo method described in the text using the same sequence of uniform random variables, but with varying $\kappa$.  Note that each realization hits the value zero only once in the third jump, and as $\kappa$ is increased, the amount of time spent at zero increases as illustrated.   
}
\label{fig:1dnumerical}
\end{figure}

\subsection{Accuracy and efficiency of the numerical method} 

We compare the accuracy and efficiency of this numerical method for simulating a SBM, to the accuracy and efficiency of a method which simulates \eqref{1d:dX} directly using a symmetrized Euler-Maruyama (EM) approximation \eqref{1d:symm_euler} and a Morse potential for $U^\epsilon(x)$. 
Figure \ref{fig:accuracy_sbm_1d} shows that 1\% accuracy for the sticky random walk requires a mean time step of about 0.1. Contrast this to the EM scheme, which, for an accuracy of about 2\%, requires a potential width no more than about $118^{-1}$ and a time step of about $1.25\times 10^{-6}$. In this example, the time step of the sticky random walk method is about \emph{five orders of magnitude larger}, to obtain a comparable level of accuracy.

In applications, one may with to use a SBM to model a potential with a small, but not infinitesimal range. How much more efficient is the sticky random walk in capturing the statistics of \eqref{1d:dX} for finite $\epsilon$, and how accurately does it do so? To test this we ran a high-resolution simulation of \eqref{1d:dX} ($\delta t = 3.91\times10^{-6}$) using a Morse potential (Example \ref{ex:morse}) with sticky parameter $\kappa = 30$, and range parameters $a\approx 30$ (well depth $D_e =  7.22$), characteristic of certain depletion interactions \cite{Meng:2010gsa,Perry:2015ku}, and $a\approx 100$ ($D_e=8.5$),  characteristic of certain DNA-mediated interactions \cite{Wang:2015ep}. 
We computed the statistics $\E X_1^\epsilon$, $\mathbb{P}(X_1^\epsilon<0.15)$, and compared them to statistics estimated using the sticky random walk method with $\kappa=30$.
Figure \ref{fig:depletion_dna_numerics} shows that the sticky random walk method gives virtually the same estimates for all time steps; in particular the largest mean time step used, $\delta t \approx 10^{-2}$, is as good as any smaller timestep. In contrast, statistics computed from a direct simulation of \eqref{1d:dX} are nowhere near their converged values until a timestep of $\delta t = 3\times 10^{-5}, 2\times 10^{-6}$ for $a=30,100$ respectively, \emph{more than 2 orders of magnitude smaller} than for the sticky random walk. 

The sticky random walk method will always contain a non-zero relative error no matter what the timestep, since it is only an approximation to the true dynamics. How large is this error? 
For the mean, $\E X_1^\epsilon$, the relative error between the sticky random walk and the high-resolution simulation is $40\%, 13\%$ for $a=30,100$ respectively, and for $\mathbb{P}(X_1^\epsilon<0.15)$, the relative errors are $6\%, 2\%$. While the relative error in computing the mean may seem high, notice that the mean is sensitive to the range and shape of the potential, since we are computing this statistic at short times so most of the probability is concentrated in the boundary layer associated with the potential. In practice, experimental measurements of a quantity that depends so sensitively on the shape of a short-ranged potential are infeasible, because measurement noise is typically larger than the width of the potential \cite{Perry:2015ku}.  In addition, the quantity of interest would usually be the particle diameter plus the mean distance, which would have a significantly smaller relative error, less than 1\% for a 1$\mu$m particle. A coarser statistic such as $\mathbb{P}(X_1^\epsilon<0.15)$ is more practical to measure, and hence is a better test for the accuracy of the sticky random walk model; here the relative errors are much smaller.


\begin{figure}[ht!]
  \begin{minipage}{\textwidth}
      \begin{minipage}[b]{0.33\textwidth}
    \includegraphics[width=\textwidth]{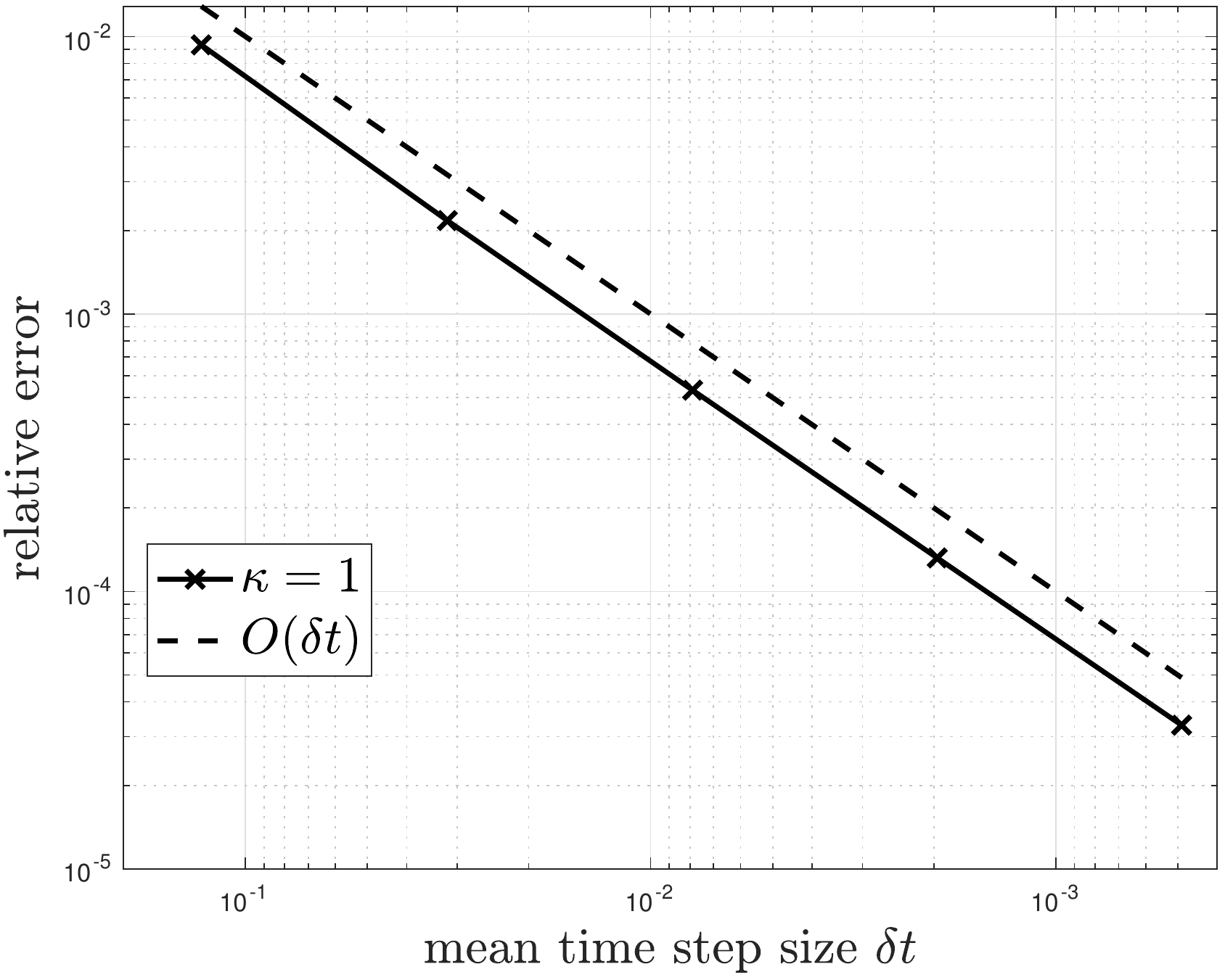} 
  \end{minipage}   \hfill
  \begin{minipage}[b]{0.63\textwidth}
  \centering
\begin{tabular}[b]{c|c|c|c}
\begin{tabular}{c} \\
\\
$\Delta t$ \end{tabular}    &  \begin{tabular}{c} $D_e=2.5$  \\ $a\approx13.6$ \\ $\epsilon\approx0.27$ \end{tabular} & \begin{tabular}{c} $D_e=3.5$ \\
$a\approx31.4$ \\
$\epsilon\approx0.18$ \end{tabular} & \begin{tabular}{c} $D_e=5$ \\
$a\approx118$ \\
$\epsilon\approx0.09$ \end{tabular} \\
\hline
$10^{-5}$     & $40\%$ & $7.6\%$ & $134.7\%$   \\
$5 \cdot 10^{-6}$     & $40\%$ & $7.2\%$ & $54.5\%$ \\
$2.5 \cdot 10^{-6}$     & $40\%$ & $7.0\%$ & $2.9 \%$ \\
$1.25 \cdot 10^{-6}$     & $40\%$ & $6.5\%$ &  $2.1 \%$ \\
$6.25 \cdot 10^{-7}$     & $40\%$ & $6.3\%$ & $3.5\%$  
\end{tabular}  
\vspace{0.0cm}
  \end{minipage} 
      \begin{minipage}[b]{0.33\textwidth}
      \centering
(a) 
\end{minipage}
\begin{minipage}[b]{0.66\textwidth}
\centering
(b) 
\end{minipage}
  \end{minipage}
\caption{ This figure compares the accuracy of (a) the sticky random walk with (b) the symmetrized Euler-Maruyama scheme in \eqref{1d:symm_euler} with $\kappa=1$ and $U^\epsilon$ a Morse potential. Accuracy is measured with respect to the local solution $u(0,1)$ of \eqref{eq:1dibvp} with initial condition $\phi(x) = \exp(-(x-3)^2)$.  As benchmark solution, we use the semi-analytical solution given in Theorem 1 of Ref.~\cite{Peskir:2014}.  In (a), we see that the sticky random walk is within $1\%$ accuracy with a (mean) time step size of about $0.1$, and is second-order accurate.   In (b), we see that  about $2\%$ accuracy requires a potential width $a^{-1} \leq 118^{-1}$ 
and a time step size $\delta t \le 1.25 \times 10^{-6}$.  
 }
\label{fig:accuracy_sbm_1d}
\end{figure}

\begin{figure}[ht!]
  \begin{minipage}{\textwidth}
      \begin{minipage}[b]{0.5\textwidth}
        \includegraphics[width=\textwidth]{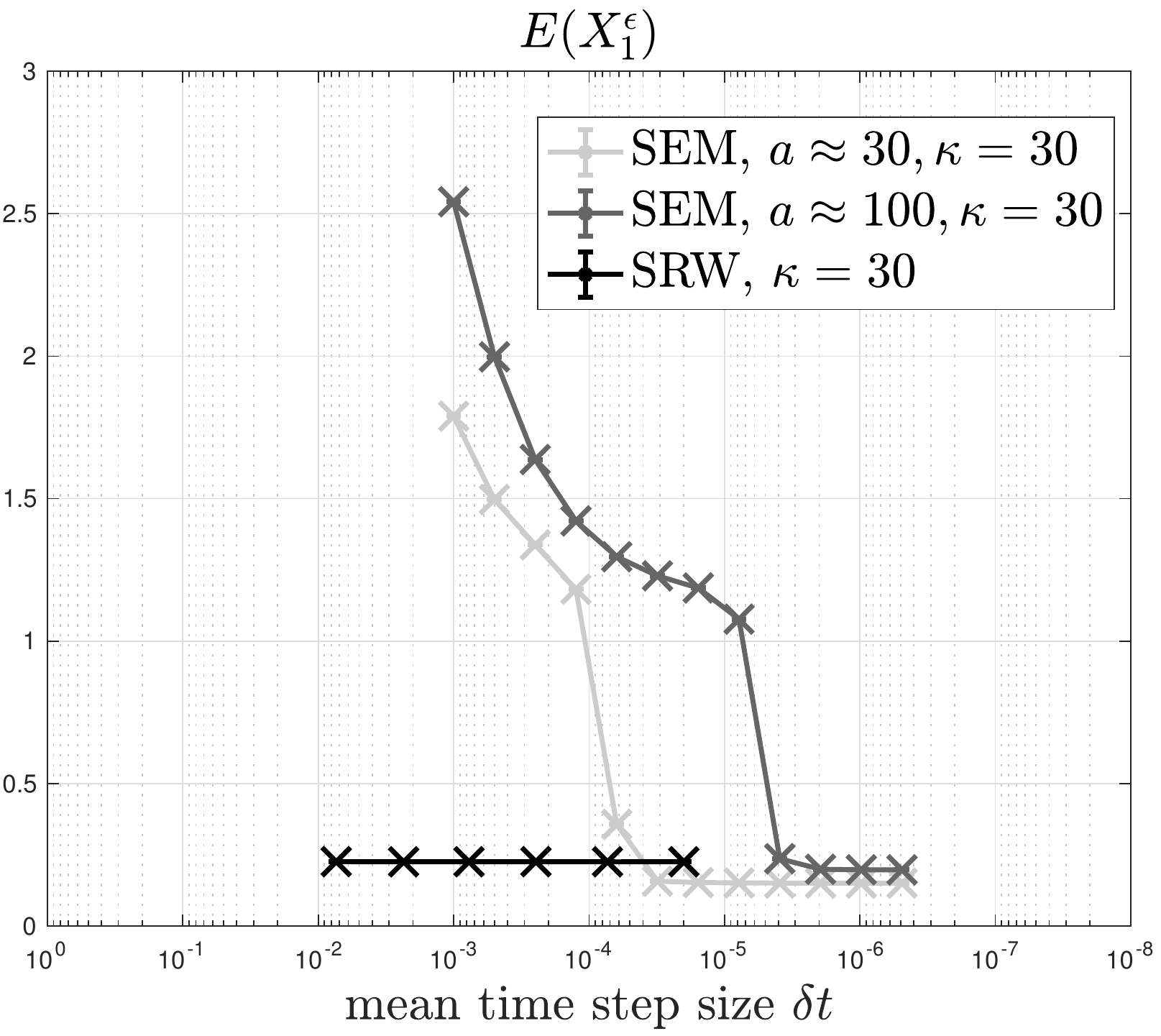} 
  \end{minipage}   \hfill
  \begin{minipage}[b]{0.5\textwidth}
      \includegraphics[width=\textwidth]{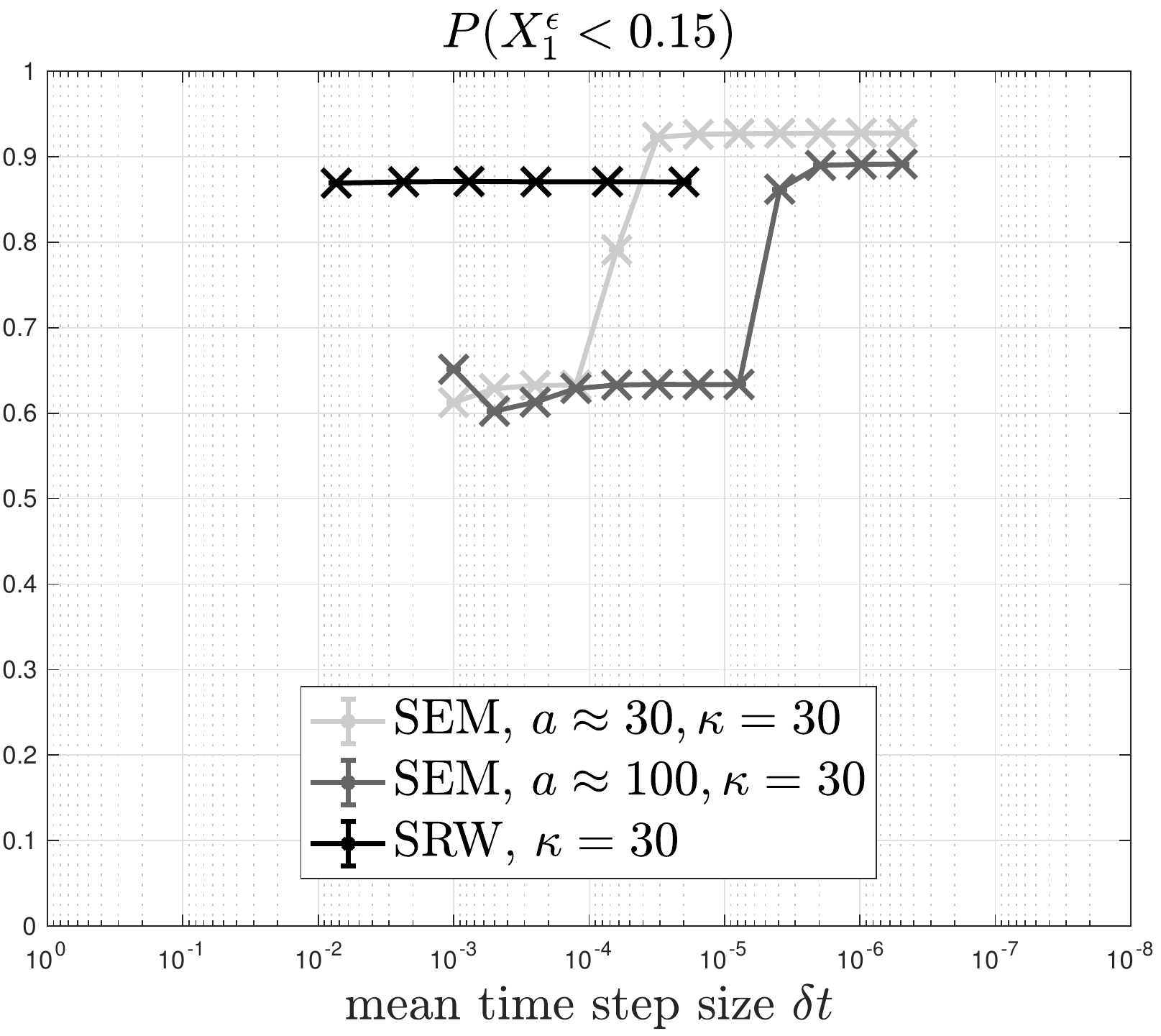} 
  \end{minipage} 
      \begin{minipage}[b]{0.5\textwidth}
      \centering
(a) 
\end{minipage}
\begin{minipage}[b]{0.5\textwidth}
\centering
(b) 
\end{minipage}
  \end{minipage}
\caption{  Panel (a) and (b) show symmetrized Euler-Maruyama (SEM) approximations of $\E(X_1^{\epsilon})$ and $P(X_1^{\epsilon}<0.15)$ respectively, at the parameter values and time step sizes indicated in the figure.  Also plotted are sticky random walk (SRW) approximations of these quantities.  For both SEM and SRW, $10^7$ samples are used.  In all cases, SRW is converged at a mean time step sizes several orders of magnitude larger than SEM, because in the sticky limit the strong short-ranged interactions are eliminated.
 }
\label{fig:depletion_dna_numerics}
\end{figure}

\subsection{PDEs that can be solved by running a sticky random walk}\label{sec:1dFK}

It should be clear from the calculations in section \ref{sec:1dnumerics} that the generator $Q$ of $Y_t$ is a locally consistent approximation of the generator $\mathcal L$ of a SBM. 
How does this local consistency connect to more global results, for example about the accuracy with which $Y_t$ can reproduce statistics of a sticky Brownian motion? 
In this section we present new Feynman-Kac formulae to show that from trajectories of $Y_t$, one can numerically solve certain PDEs with Feller's more general boundary condition to second-order accuracy in the space step.  


Consider the heat equation with Feller's boundary condition 
\begin{equation} \label{eq:1dibvp}
\begin{dcases}
\partial_t u(x,t) = \partial_{xx} u(x,t) & x \ge 0, t \ge 0 \\
u(x,0) = \phi(x) & x \ge 0 \\
p_1 u(0, t) - p_2 \partial_x u(0,t) + p_3  \partial_{xx} u(0, t) = 0 &  t \ge 0
\end{dcases}
\end{equation}
where $\phi: [0, \infty) \to \mathbb{R}$ and $p_1, p_2, p_3$ are parameters satisfying 
$p_1 + p_2 + p_3 = 1$ and $p_1, p_2, p_3 \ge 0$ \cite{Feller:1952te,Feller:1957ua,Ito:1963wl}. 
Special cases of this boundary condition include: 
\begin{description}
\item[1) $p_1=1$:] Dirichlet b.c. corresponding to a stopped Brownian motion;
\item[2) $p_2=1$:] Neumann b.c. corresponding to a reflecting Brownian motion;
\item[3) $p_3=1$:] Wentzell b.c. corresponding to an absorbed Brownian motion;\footnote{An absorbed Brownian motion is one where the Brownian motion has a fixed point at the boundary, which is consistent with what happens as the sticky parameter in \eqref{1d:FPlim} becomes infinite; whereas, a stopped Brownian motion corresponds to a Brownian motion that is terminated at the boundary.} and,
\item[4) $p_3=0$:] Robin b.c. corresponding to an elastic Brownian motion.
\end{description}
The general case corresponds to a Brownian motion that exhibits stopping, stickiness, and reflection at various rates \cite{Ito:1963wl}. 

The following result shows how to numerically solve \eqref{eq:1dibvp} when $p_2 > 0$ using the sticky random walk developed in \S\ref{sec:1dnumerics}. The case $p_2 = 0$ can be solved in a similar manner using a continuous-time random walk with absorbing boundary conditions at $0$, but for notational brevity, we omit this case.  

\begin{theorem} \label{thm:sbm_ibvp}
Assume $p_2 >0$, set 
\[
\kappa = p_3 / p_2\,, \qquad c =   -2 p_1 / (h p_2 + 2 p_3) \,.
\] Suppose further that \eqref{eq:1dibvp} has a solution $u(x,t)$ whose first four derivatives are continuous and bounded.   Let $Y_t$ be a sticky random walk, i.e., a Markov jump process on the grid $\mathcal R^h$ with generator $Q$ defined by \eqref{1d:Ldisc}.
At every grid point $x_i \in \mathcal R^h$, define the  function \begin{equation} \label{eq:uh}
u^h_i(t) = 
\E_{x_i} \left( \phi(Y_t) \exp\left(c \int_0^t 1_{ \{ 0 \} }(Y_s) ds \right) \right) \;, 
\end{equation}
where $\E_{x_i}$ denotes expectation conditional on $Y_0 = x_i$.  Then for all $T>0$ and for all $h>0$ sufficiently small, there exists $C(T)>0$ such that \[
\sup_{i \ge 0, t \in [0, T]} |u^h_i(t) - u(x_i,t)| \le C(T) h^2 \;.
\]
\end{theorem}

Although there are well-posedness and semigroup generation results for the heat equation with generalized, two-sided Feller boundary conditions in a bounded domain \cite{MR1695147,favini2002heat}, we are not aware of results for the heat equation on the half line with Feller boundary conditions that gives the required regularity results for $u(x,t)$ in Theorem~\ref{thm:sbm_ibvp}.

\begin{proof}
Our proof will proceed in two steps. First, we show that a second-order spatial discretization of the PDE \eqref{eq:1dibvp} using finite differences is given by the
 solution $u^h(t)\in \mathcal R^h$ to the linear, infinite-dimensional system of ODEs 
\begin{equation} \label{eq:odesM}
\dot u^h_i(t) =  \begin{dcases} \frac{u^h_{i+1}(t) - 2 u^h_{i}(t) + u^h_{i-1}(t)}{h^2}  &  i = 1, 2, 3, ... \\
\frac{2 u^h_1(t) - 2 u^h_0(t)}{h^2 + 2 h \kappa}  + c u^h_0(t) & i = 0
\end{dcases} 
\end{equation}
with initial condition $u^h(0) = \left. \phi \right|_{ \mathcal R^h}$. 
Together with the regularity of $u(x,t)$, this will imply that $u^h_i(t)$ and $u(x_i,t)$ are close in an $\ell_{\infty}$ norm over $\mathcal R^h$.  Second, we show that the solution $u^h_i(t)$ to \eqref{eq:odesM} admits the stochastic representation \eqref{eq:uh}.   
 
 
Note that another way to write \eqref{eq:odesM} is to define $u^h(x_i,t)=u_i^h(t)$, and write 
 \begin{equation} \label{eq:odes2M}
  \partial_t u^h(x_i, t) = Q u^h(x_i, t) + c  u^h(x_i, t) 1_{\{ 0 \}}(x_i) \;, \quad u^h(x_i, 0) = \phi(x_i) \;,
\end{equation}
where $Q$ is the generator of $Y_t$ defined in \eqref{1d:Ldisc}. The difference between \eqref{eq:odes2M} and the backward equation associated with $Y_t$  is the presence of the term $cu_0^h$, which causes $u(0,t)$ to decrease in time (recall $c<0$).

For the first step, i.e. obtaining the discretization  of the PDE \eqref{eq:1dbvp}, notice that the ODEs \eqref{eq:odesM} for $k \ge 1$ are obtained by replacing $\partial_{xx}$ in \eqref{eq:1dibvp} by the centered finite difference given in \eqref{1d:Lfdisc}; and the ODE at $k=0$ is obtained by incorporating a finite difference discretization of the boundary condition into a centered finite difference approximation of $\partial_{xx}$ at $0$.  
Specifically, given a function $f$ that satisfies the boundary condition in \eqref{eq:1dibvp}, use a ghost grid value $f_{-1}$ combined with a centered finite difference approximation to $f'(0)$ and $f''(0)$ to obtain
\[
p_1 f_0 - p_2 \frac{f_1 - f_{-1}}{2 h} + p_3  \frac{f_{1} - 2f_0 + f_{-1}}{h^2}  = O(h^2)  \;.
\]
Using the above equation to eliminate $f_{-1}$ from \eqref{1d:Lfdisc} when $k=0$, gives an approximation to $\partial_{xx}$ at the boundary as
\begin{equation}
\partial_{xx} f (0) =  \frac{p_2}{h (h p_2 + 2 p_3)} ( f_1 - f_0)  + c f_0 + O\left(\frac{h^2}{h p_2 + 2 p_3} \right) \;.
\end{equation}
Substituting $p_3 = \kappa p_2$ into the first term then gives the ODE in \eqref{eq:odesM} for $k=0$.  

For the second step, 
define
\[
M_s = u^h(Y_s, t - s) e^{c  \int_0^s 1_{ \{ 0 \} }(Y_r) dr}\,.
\]
Applying It\^o's formula for jump processes (see e.g.~\cite{protter2004stochastic}) we obtain that
\begin{align*}
 \E M_s -\E M_0&= \E \int_0^s \Big(  -\partial_t u^h(Y_r, t-r)  + Q u^h(Y_r, t-r)  \\
 & \qquad \qquad 
 + c u^h(Y_r, t-r)1_{ \{ 0 \} }(Y_{r})  \Big) e^{c  \int_0^r 1_{ \{ 0 \} }(Y_{r'}) dr'} dr\,.
\end{align*} 
By \eqref{eq:odes2M}, the right-hand side is zero. Evaluating this equation at $s=t$ and using the initial condition in \eqref{eq:1dibvp} gives 
\[
u^h(x_i, t) = \E_{x_i} \left( \phi (Y_t) e^{c  \int_0^t 1_{ \{ 0 \} }(Y_r) dr} \right) \;,
\] 
as required.
\end{proof}

One can similarly solve Dirichlet-Poisson problems.  For example, consider the boundary
value problem
\begin{equation} \label{eq:1dbvp}
\begin{dcases}
\partial_{xx} u(x) = - \phi(x)  \;, & x \in (0, \ell) \\
p_1 u(0) - p_2 \partial_x u(0) + p_3  \partial_{xx} u(0) = 0  \;, &
u(\ell) = 0 \;,
\end{dcases}
\end{equation}
where $\phi: [0, \ell] \to \mathbb{R}$ is a given function and $p_1, p_2, p_3$ are parameters satisfying 
$p_1 + p_2 + p_3 = 1$ and $p_1, p_2, p_3 \ge 0$.  Then we have the following result.

\begin{theorem} \label{thm:sbm_bvp} 
Assume $p_2 >0$, set $\kappa = p_3 / p_2$ and $c =   -2 p_1 / (h p_2 + 2 p_3) $; and suppose that \eqref{eq:1dbvp} has a solution $u(x)$ whose first four derivatives are continuous and bounded.
Let $Y_t$ be a sticky random walk, i.e., a Markov jump process on the grid $\mathcal R^h$ with generator $Q$ defined by \eqref{1d:Ldisc}.
For every grid point $x_i \in \mathcal R^h \cap [0, \ell]$, define the function \begin{equation} \label{eq:uh2}
u^h_i = \E_{x_i} \left( \int_0^{\tau^h} \phi(Y_t) \exp\left( c \int_0^t 1_{ \{ 0 \} }(Y_s) ds \right) dt \right) \;, 
\end{equation}
where $\E_{x_i}$ denotes expectation conditional on $Y_0 = x_i$ and $\tau^h$ is the random stopping time $\tau^h = \inf\{ t \ge 0 \mid Y_t = \ell\}$.  Then for all $h>0$ sufficiently small, there exists $C>0$ such that \[
\sup_{x_i \le \ell} |u^h_i - u(x_i)| \le C h^2 \;.
\]
\end{theorem}

A proof of this theorem is like the proof of Theorem~\ref{thm:sbm_ibvp} and therefore omitted.

\section{Other characterizations of sticky Brownian motion}\label{sec:1dprobability}

Interestingly, sticky boundary conditions were discovered in the pure mathematics literature, well before they were used in applied math or physics.  This section summarizes some of the other ways of characterizing SBM that have been developed in the mathematical literature, and connects them to our numerical and asymptotic approach. A nice history focusing especially on Feller's contribution to the development is given in \cite{Peskir:2015be}. 
To more easily make a connection to the existing literature, we consider in this section a traditional sticky Brownian motion, whose generator is $\mathcal L^0f=\half \partial_{xx}f$ with boundary condition $\partial_x f|_{x=0} = \kappa \partial_{xx}f|_{x=0}$. 
The numerical solution of this process is the same as that given in section \ref{sec:1dnumerics}, but with numerical generator $Q^0 = Q/2$, where $Q$ is the generator for a root-2 sticky Brownian motion, defined in \eqref{1d:Ldisc}. 

Sticky diffusion processes were discovered by Feller in the 1950s, who sought to identify the most general 
boundary conditions for the generator of a one-dimensional diffusion  \cite{Feller:1952te,Feller:1957ua}. For a process on $[0,\infty)$ with continuous sample paths that behaves like a Brownian motion in $(0,\infty)$, the possible boundary conditions have the form \cite{Feller:1952te,Ito:1963wl,Peskir:2015be}
\begin{equation} \label{eq:most_general_bc}
p_1f(0) - p_2f'(0) + p_3f''(0) = 0\,,
\end{equation}
with $p_1,p_2,p_3 \geq 0$, and $p_1+p_2+p_3=1$ (c.f.~\eqref{eq:1dibvp}).
Feller said of his result \cite{Feller:1952te}: 
\begin{quote}
This is the second instance of a concrete problem with physical significance where physical intuition failed, but our abstract methods provided a clue.
\end{quote}
Feller was referring to problems in genetics, which have certain singularities in the diffusion coefficients when a population size hits zero, but his remark could equally well apply to the sticky boundary condition $f'(0) = \frac{p_3}{p_2}f''(0)$; physical intuition does not immediately show why this should be associated with stickiness, and a singularity in the transition probabilities.

Subsequently It\^{o} \& McKean  showed how to construct 
an SBM from a time change of a reflecting Brownian motion \cite{Ito:1963wl}. 
Let $B_t^+$ be a standard reflecting Brownian motion, which can be constructed from a standard Brownian motion $B_t$ as $B^+_t = |B_t|$. Let 
\[
l^0_t(B^+) = \lim_{\epsilon\to 0}\frac{1}{2\epsilon}\int_0^t1(B_s^+ < \epsilon) ds
\] be the local time accumulated at $0$ by $B^+$ over $[0,t]$.  Define
\[
A_t = t+ 2\kappa\: l^0_t(B^+) \;.
\]
 Since $l^0_t(B^+)$ is continuous and nondecreasing, $A_t$ is continuous and strictly increasing, so it has an inverse $T(t) = A^{-1}(t)$.  Define
\begin{equation}\label{1d:Bstar}
B^*_t = B^+_{T(t)}\,.
\end{equation}
Then $B^*_t$ is an SBM \cite{Ito:1963wl}.   

The process $B^*_t$ is simply a time change of a reflecting Brownian motion, making it run on clock $T(t)$ instead of clock $t$. When $B^*_t\neq 0$, then $dA/dt= dT/dt= 1$ so the process runs at the same rate as $B^+_t$. When $B^*_t=0$, then $A(t)$ increases faster than $t$ because of the local time, so the clock $T(t)$ runs more slowly and $B^*_t$ slows down at $0$. 

To see why it slows down in a way that changes its measure at zero, let's compute the MFPT for $B^*_t$ to leave the interval $[0,h]$, starting at 0. 
Let $\tau = \inf\{ t>0 : B_t^+ > h \}$ be the first exit time of $B_t^+$ from $[0,h]$, so that $\E( \tau) =h^2$ and the MFPT of $B_t^*$ is given by  \[
\E( A_{\tau} ) = h^2 + 2 \kappa \E( \ell_\tau^0(B^+) ) = h^2 + 
2 \kappa \lim_{\epsilon \to 0} \frac{1}{2 \epsilon} \E( \int_0^\tau1(B_s^+ < \epsilon) ds ) \;.
\]  
By a Feynman-Kac formula, the expected value $u(x)=\E_x( \int_0^\tau1(B_s^+ < \epsilon) ds )$ can be calculated by solving the boundary value problem \[
\frac{1}{2} u''(x) = - 1_{[0,\epsilon]}(x) \;, ~~ x \in [0, h] \;, ~~ u'(0)=0,~~ u(h)=0 \;, \] 
whose local solution at zero is $u(0) = 2 \epsilon h - \epsilon^2$.  Hence, 
$\E( \ell_{\tau}^0(B) ) = h$ and $\E( A_{\tau} ) = h^2 + 2 \kappa h$, which is twice the mean waiting time we derived for a root-2 SBM in Lemma \ref{l:mfpt_sbm}, and in particular, this waiting time is \emph{not} diffusive like the mean waiting time of $B_t^+$.

A different discrete approximation to a SBM was derived by Amir \cite{Amir:1991wb}, who showed how to obtain the process as a limit of random walks. Recall that one can construct a standard Brownian motion starting from a random walk by using Donsker's Theorem \cite{donsker1951}. Let 
\[
S_{\Delta x}(t) = \Delta x\sum_{i=1}^k X_i\quad \text{for}\quad k\Delta t \leq t < (k+1)\Delta t
\]
be a rescaled random walk, 
where $\{X_i\}_{i=1}^\infty$ is a sequence of independent, identically distributed random variables with $P(X_i=1)=P(X_i=-1) = \frac{1}{2}$.
As $\Delta x\to 0$ with $\Delta t = (\Delta x)^2$, the process $S_{\Delta x}(t)$ converges in distribution to a Brownian motion $B_t$ \cite{donsker1951,Feller:BookVol2}. 

Amir \cite{Amir:1991wb} showed that one can construct a sticky (nonreflecting) Brownian motion, by modifying $S_{\Delta x}$ as follows: every time $S_{\Delta x}(t)$ hits zero, wait there for a time interval of length $\sqrt{\Delta t}$, instead of $\Delta t$. For example, one can let $\Delta x = 2^{-n}$, $\Delta t = 2^{-2n}$ for some integer $n$, and then the process must wait $2^n$ timesteps each time it hits zero. As $n\to \infty$, the modified process $S^*_{n}(t)$ converges to a sticky nonreflecting Brownian motion. A sticky reflecting Brownian motion is then the limit of $|S^*_{n}(t)|$.

Finally, we remark that an SBM $B^*_t$ solves the following 
SDE, which should be interpreted in an integrated sense  (see \cite{IkWa1989, Engelbert:2014gg}, and references therein): 
\begin{equation}\label{sbm_SDE}
dB^*_t = \frac{1}{2\kappa}1(B^*_t=0)dt + 1(B^*_t>0)dB_t\,.
\end{equation}
This equation has a unique weak solution, but no strong solution \cite{Engelbert:2014gg}. 
That it corresponds to the sticky boundary condition can be seen by noting that\footnote{This perspective is due to R. Varadhan, personal communication.} the generator away from the origin is $\mathcal A f = \half \partial_{xx} f$, and at the origin it is $\mathcal A f = \frac{1}{2\kappa}\partial_x f$, so to be consistent the function $f$ must satisfy the sticky boundary condition $\half \partial_{xx} f = \frac{1}{2\kappa}\partial_x f$. 

It may seem quite remarkable that changing the diffusion coefficient at a single point, the origin, can have such a dramatic effect on the process. To see why this is so, consider widening the region near the origin by some amount $\epsilon\ll 1$,  approximating the drift as constant in this interval, and estimating 
 the increment $\Delta B_t^*$ over a small time interval $\Delta t$. 
If $B^*_t < \epsilon$, the increment is $\Delta B^*_t\approx \frac{1}{2\kappa}\Delta t$, so it takes the process a time of about $\Delta t \approx2\kappa \epsilon$ to leave the interval $[0,\epsilon)$. 
If $B^*_t > \epsilon$, the increment has magnitude
 $|\Delta B^*_t|\approx |\Delta B_t| \approx \sqrt{\Delta t}$, so it takes the process a time of about $\Delta t \approx \epsilon^2$ to travel the length of any other interval of length $\epsilon$. We recover the same scalings as for the time-changed formula \eqref{1d:Bstar}, showing that the slowdown near the origin has a singular effect on its transition densities. Even though the diffusion coefficient changes at only one point, the difference between $dt$ and $dB_t$ at that point gives rise to the singular change in timescales.



\section{Conclusion and Outlook} \label{sec:conclusion}

We considered a reflecting Brownian motion on a half-line with a deep but short-ranged potential energy near the origin. As the potential becomes narrower and deeper, the process approaches a ``sticky'' Brownian motion,  which has finite probability to be found in any interval on the half-line, as well as finite probability to be found exactly at $\{0\}$; counterintuitively it never spends an interval of time at 0. The process is characterized by a non-classical boundary condition for its generator, involving second derivatives. 
We showed that simulating trajectories of a process that is close to sticky using a traditional Euler-Maruyama discretization of its SDE requires very small timesteps. This motivated us to introduce an alternative method, based on discretizing space first and constructing a continuous-time Markov chain on the set of discretization points, which allows time steps at least 2 orders of magnitude larger in parameter regimes of physical relevance (at the expense of making a small error in estimating the transition probability.) The method results in a random walk on the nonnegative integers with random holding times, with a larger holding time at $0$. The holding times give insight into why the process has a singular probability density at zero, since the holding time scales ballistically with step size at the origin, rather than diffusively. 

Our motivation came from studying systems of mesoscale particles (diameters $\approx 100$nm-$10\mu$m), which have attractive interactions that are very short-ranged compared to their diameters. For such particles, the structures and dynamics are often more effectively studied by considering the system in the sticky limit \cite{HolmesCerfon:2013jw,Perry:2015ku,HolmesCerfon:2017hz}. Our goal is to eventually create numerical methods to simulate such particle systems directly in the sticky limit, for which we hope to see a similarly large gain in efficiency. 


Achieving this goal will require building on the ideas in this paper to handle higher-dimensional sticky diffusions with more complicated boundary conditions. There are several steps involved. The first is to create methods to handle $d$-dimensional diffusions that are sticky on a half-space of dimension $d-1$. The new ingredient here is that the diffusion can move directly along the boundary, with dynamics that are different from those in the interior. This will be feasible by discretizing the generators in the domain and on its boundary, as we have done in this paper. We expect such a method to be efficient even when $d$ is too large to solve PDEs numerically, since the generator must only be discretized locally, using two points per dimension, and not globally over the whole domain \cite{BoVa2018}.

A second step is to adapt these methods to work on manifolds, since a collection of interacting sticky particles performs a diffusion on the manifold corresponding to certain fixed distance constraints \cite{HolmesCerfon:2013jw}. It may be a challenge to retain the method's second-order accuracy when on a manifold, however it may not be necessary to do so, since the sticky limit is an approximation anyways. Other methods have considered how to sample probability densities directly on manifolds and have shown themselves to be significantly more efficient than using short-range forces to keep a process near a manifold (e.g. \cite{Ciccotti:2007fv,Lelievre:2012id,Zappa:2017ub}.)

A third step, and perhaps the most challenging, will be to adapt the method to processes which are sticky on even lower-dimensional ``corners''. For example, a $d$-dimensional process may stick to a $d-1$-dimensional boundary, from which it may stick to a $d-2$-dimensional boundary, and so on. Physically this could correspond to a system of particles forming one bond, then two, then three, and so on. The number of intersecting boundaries increases as one moves down in dimension, and dealing with the whole collection of boundaries together may require further approximation.

We have focused in this paper exclusively on one-dimensional sticky Brownian motion, in order to build intuition into this unusual process, and to introduce it to the applied math community. Along the way we discussed some of the connections to the probabilistic approaches to studying sticky diffusions. We hope these ideas will be useful for other researchers studying sticky diffusions in the myriad of contexts in which they may arise, from biology to materials science to finance to operations research, and other applications yet to be imagined.


\appendix

\section{Listings} \label{sec:listings}

In the MATLAB file \texttt{SEM.m} in Listing~\ref{list:SEM}, we apply the symmetrized Euler-Maruyama scheme to the SDE \eqref{1d:dX} with $U^{\epsilon}(x)$ being a Morse potential energy function (see Example~\ref{ex:morse}) with parameters $D_e=8.5$, $a=\sqrt{\pi}e^{D_e} /(\kappa \sqrt{D_e})$, and $x_0=1/a$; and the sticky parameter set at $\kappa=30$. A single realization of a discretized Brownian motion is produced over $[0,1]$ with $\delta t=10^{-6}$. 

This realization of Brownian motion is used to drive  symmetrized Euler-Maruyama operated at two time step sizes: $\delta t$ and $R \delta t$ where $R=16$ \cite{Hi2001}.  We set the seed for MATLAB's random number generator, arbitrarily, to be 88 using the command \texttt{rng}.  The plotted paths illustrate that the time step size of symmetrized Euler-Maruyama has to be sufficiently small in order for the method to accurately represent the strong, short-ranged Morse potential force.

Listing~\ref{list:SRW} displays the MATLAB file \texttt{SRW.m} which produces a realization of the sticky random walk over the time interval $[0,T]$ with initial condition $x_0=0.2$, sticky parameter $\kappa=30$, and spatial step size $h=0.01$.  The simulation is terminated after the time update first exceeds $T$.  We set the seed for MATLAB's random number generator, arbitrarily, to be 999 with the command \texttt{rng}.  A single sample paths is plotted as a stairstep graph using the command \texttt{stairs}.

\begin{lstlisting}[language=Matlab,frame=single,basicstyle=\ttfamily\scriptsize,float, caption={Symmetrized Euler-Maruyama: {\tt SEM.m}}, label=list:SEM]
%SEM Symmetrized Euler-Maruyama method applied to IVP
%       dX = f(X) dt + sqrt(2) dW, X(0) = X0 >= 0
% with a reflection b.c. at the origin where f is defined below.
%
% Discretized Browian motion uses timestep dt.
% SEM uses timestep dt or R*dt.

rng(88);
kappa=30.0; % sticky parameter
De=8.5;     % well depth
a=sqrt(pi)/kappa*exp(De)*1.0/sqrt(De);  % inverse potential range
r0=1/a;                                 % minimizer of Morse potential
% Morse potential force with parameters De, a, and r0
f=@(r) (r<1/sqrt(a)).*(-2.0*a*De*exp(-a*(r-r0))*(1.0-exp(-a*(r-r0)))); 

T=1;        % time interval of simulation
dt=1e-6;    % small time step size
nsteps=ceil(T/dt);

dW=sqrt(2.0*dt)*randn(nsteps,1);    % discretized Brownian increments 
X=zeros(nsteps,1); t=zeros(nsteps,1);
X(1)=0.2; t(1)=0.0;
for i=2:nsteps+1
    t(i) = t(i-1) + dt;
    X(i) = abs(X(i-1) + dt*f(X(i-1)) + dW(i-1));
end

R=16; Dt=R*dt; 
XX=zeros(nsteps/R,1); tt=zeros(nsteps/R,1);
XX(1)=0.2; tt(1)=0.0;
for i=2:nsteps/R+1
    dWW=sum(dW(R*(i-2)+1:R*(i-1)));
    tt(i) = tt(i-1) + Dt;
    XX(i) = abs(XX(i-1)+ Dt*f(XX(i-1)) + dWW);
end

figure(1); clf; hold on;
plot(t,X,'k', 'LineWidth',2);
plot(tt,XX,'color',[0.6 0.6 0.6], 'LineWidth',2);
lh=legend({['$\delta t=$' num2str(dt,'%7.6f')], ...
    ['$\delta t=$' num2str(Dt,'%7.6f')]}, ...
    'location', 'best', 'Interpreter','latex', 'fontsize',20);
set(lh,'fontsize',20,'NumColumns',2);
xlabel('$t$','fontsize',20,'Interpreter','latex');
ylabel('$X$','fontsize',20,'Interpreter','latex',...
    'Rotation',0,'HorizontalAlignment','right');
set(gca,'Ytick',0:1:3,'FontSize',20);
set(gca,'Xtick',[0 1/2 1],'xticklabel',{'0', '1/2', '1'},'FontSize',20);
xlim([0 T]); ylim([0 3]);

\end{lstlisting}

\begin{lstlisting}[language=Matlab,frame=single,basicstyle=\ttfamily\scriptsize,float, caption={Sticky Random Walk: {\tt SRW.m}}, label=list:SRW]
%SRW Sticky random walk method with sticky parameter kappa
%    and spatial step size h.
%

clc;
rng(999);

kappa=30;
T=1;
h=0.01; h2=h*h;
x0=0.2;
Y=[x0]; t=[0]; i=1;
while (1)
    u0=rand; u1=rand;    
    if (Y(i)<0.1*h) 
        mean_dt=0.5*h2+kappa*h; 
        Y(i+1)=Y(i)+h;
    else
        mean_dt=0.5*h2;
        gamm=(u1<0.5);
        Y(i+1)=Y(i)+h*gamm-h*(1-gamm);
    end
    dt=-log(u0)*mean_dt;
    t(i+1)=t(i)+dt;
    
    if (t(i+1)>T) 
        break; 
    end
    i=i+1;
end

figure(1); clf; hold on;
stairs(t,Y,'k', 'LineWidth',2);
xlabel('$t$','fontsize',20,'Interpreter','latex');
ylabel('$Y_t$','fontsize',20,'Interpreter','latex',...
    'Rotation',0,'HorizontalAlignment','right');
set(gca,'Ytick',0:1:3,'FontSize',20);
set(gca,'Xtick',[0 1/2 1],'xticklabel',{'0', '1/2', '1'},'FontSize',20);
xlim([0 T]); ylim([0 1]);

\end{lstlisting}

\section*{Acknowledgments}
We wish to acknowledge Robert Kohn, Eric Vanden-Eijnden, and Srinivasa Varadhan for useful discussions that helped to improve this paper.

\bibliographystyle{siamplain}
\bibliography{nawaf,miranda}

\end{document}